\documentclass[12pt]{article}
\usepackage{graphicx}
\usepackage{fullpage}
\usepackage{amsmath}
\usepackage{amsthm}
\usepackage{hyperref}
\usepackage{epstopdf} 
\usepackage{mathptmx} 
\usepackage{amssymb,enumerate,dsfont,mathtools}
\usepackage{enumitem}
\usepackage[bbgreekl]{mathbbol}
\usepackage{cleveref}
\DeclareMathAlphabet{\mathcal}{OMS}{cmsy}{m}{n}
\usepackage{tikz}
\usepackage{subfig}
\crefname{ALC@unique}{Line}{Lines}
\usetikzlibrary{arrows,intersections}
\usepackage{pgfplots}
\pgfplotsset{compat=1.12}

\usepackage{color}
\usepackage{graphicx}
\usepackage{algorithm}
\usepackage{algpseudocode}

\newcommand{\abs}[1]{\left|#1\right|}

\newcommand{\tends}{\rightarrow}
\newcommand{\smnorm}[1]{\lVert#1\rVert}
\newcommand{\norm}[1]{\left\|#1\right\|}

\newcommand{\supp}{{\rm supp} \,}
\newcommand{\half}{\frac{1}{2}}
\newcommand{\ie}{i.e.\ }

\newcommand{\tst}{\textstyle}

\newcommand{\m}[1]{\mathcal{#1}}
\newcommand{\bb}[1]{\mathds{#1}}

\newcommand{\eps}{\varepsilon}

\newcommand{\R}{\bb R}
\newcommand{\N}{\bb N}

\newcommand{\p}{\partial}

\renewcommand{\theta}{{\vartheta}}
\renewcommand{\O}{{\Omega}}
\newcommand{\oO}{\overline{\O}}
\newcommand{\pO}{\p \O}
\newcommand{\calT}{\mathcal{T}}

\newcommand{\tiab}[1]{\tilde{#1}_i^{(\a,\b)}}
\newcommand{\ttiab}[1]{\tilde{\tilde{#1}}_i^{(\a,\b)}}
\newcommand{\si}[1]{{\mathsf{#1}}_i}

\newcommand{\siab}[1]{{\mathsf{#1}}_i^{(\alpha,\beta)}}
\newcommand{\siabo}[1]{{\mathsf{#1}}_{\O, i}^{(\alpha,\beta)}}
\newcommand{\siabr}[1]{{\mathsf{#1}}^{(\alpha,\beta)}_{\pO, i}}
\newcommand{\sihahb}[1]{{\mathsf{#1}}_i^{(\hat{\alpha},\hat{\beta})}}

\newcommand{\siw}[1]{{\mathsf{#1}}_i^w}
\newcommand{\sikvi}[1]{{\mathsf{#1}}_i^{k,v_i}}

\newcommand{\siwk}[1]{{\mathsf{#1}}_i^{k,w}}

\newcommand{\hil}[1]{\hat{#1}_i^\ell}
\newcommand{\ika}[1]{{#1}_i^\kappa}
\newcommand{\ila}[1]{{#1}_i^\lambda}
\newcommand{\il}[1]{{#1}_i^\ell}
\newcommand{\ik}[1]{{#1}_i^k}

\newcommand{\iko}[1]{{#1}_i^{k+1}}

\newcommand{\iab}[1]{{#1}_i^{(\alpha,\beta)}}

\newcommand{\ihahb}[1]{{#1}_i^{(\hat{\alpha},\hat{\beta})}}
\newcommand{\hahb}[1]{{#1}^{(\hat{\alpha},\hat{\beta})}}

\newcommand{\bia}[1]{\bar{#1}_i^\alpha}
\newcommand{\biab}[1]{\bar{#1}_i^{(\alpha,\beta)}}

\newcommand{\bbia}[1]{\bar{\bar{#1}}_i^\alpha}
\newcommand{\bbiab}[1]{\bar{\bar{#1}}_i^{(\alpha,\beta)}}

\newcommand{\bihahb}[1]{\bar{#1}_i^{(\hat{\alpha},\hat{\beta})}}
\newcommand{\bbihahb}[1]{\bar{\bar{#1}}_i^{(\hat{\alpha},\hat{\beta})}}
\newcommand{\bbivi}[1]{\bar{\bar{#1}}_i^{(\alpha^{k, \ell}(v_i),\beta^{k, \ell}(v_i))}}

\newcommand{\bial}[1]{\bar{#1}_i^{(\alpha, \beta),\ell}}
\newcommand{\bbial}[1]{\bar{\bar{#1}}_i^{(\alpha,\beta),\ell}}
\newcommand{\Id}{{\sf Id}}

\renewcommand{\P}{{P_i}}

\newcommand{\infnorm}[1]{\smnorm{#1}_{L^\infty(\O)}}

\newcommand{\supiab}{\sup_{\substack{i \in \N \\ (\a,\b) \in A \times B}}}
\newcommand{\infiab}{\inf_{\substack{i \in \N \\ (\a,\b) \in A \times B}}}

\newtheorem{assumption}{Assumption}[section]
\newtheorem{definition}{Definition}[section]
\newtheorem{example}{Example}[section]
\newtheorem{lemma}{Lemma}[section]
\newtheorem{theorem}{Theorem}[section]

\renewcommand{\a}{\alpha}
\renewcommand{\b}{\beta}

\newcommand{\pair}[2]{\langle #1,#2 \rangle}
\newcommand{\RN}{\mathds{R}^n }

\makeatletter
\newsavebox{\@brx}
\newcommand{\llangle}[1][]{\savebox{\@brx}{\(\m@th{#1\langle}\)}%
  \mathopen{\copy\@brx\kern-0.5\wd\@brx\usebox{\@brx}}}
\newcommand{\rrangle}[1][]{\savebox{\@brx}{\(\m@th{#1\rangle}\)}%
  \mathclose{\copy\@brx\kern-0.5\wd\@brx\usebox{\@brx}}}
\makeatother

\begin{document}

\title{Finite Element Methods for Isotropic Isaacs Equations with Viscosity and Strong Dirichlet Boundary Conditions\thanks{Bartosz Jaroszkowski (B.Jaroszkowski@sussex.ac.uk) acknowledges the support of the EPSRC grant 1816514. Max Jensen (M.Jensen@sussex.ac.uk) acknowledges the support of the Dr Perry James Browne Research Centre.}}

\author{Bartosz Jaroszkowski \and Max Jensen}

\date{Department of Mathematics, University of Sussex, Brighton, UK}

\maketitle

\begin{abstract}
We study monotone P1 finite element methods on unstructured meshes for fully non-linear, degenerately parabolic Isaacs equations with isotropic diffusions arising from stochastic game theory and optimal control and show uniform convergence to the viscosity solution.  Elliptic projections are used to manage singular behaviour at the boundary and to treat a violation of the consistency conditions from the framework by Barles and Souganidis by the numerical operators. Boundary conditions may be imposed in the viscosity or in the strong sense, or in a combination thereof. The presented monotone numerical method has well-posed finite dimensional systems, which can be solved efficiently with Howard's method. 
\end{abstract}

\section{Introduction}

Hamilton--Jacobi--Isaacs equations, in short Isaacs equations, arise from stochastic game theory and optimal control, in particular from stochastic two player zero-sum games \cite{souganidis_twoplayer_viscosity}. The solution of the Isaacs equation corresponds under appropriate assumptions to a value function of the game. The equations are of the structure
\begin{align} \label{eq:isaacsintro}
-\partial_t v + \adjustlimits \inf
_\beta \sup_\alpha (L^{(\a,\b)} v - f^{(\a,\b)}) = 0,
\end{align}
where $\inf \sup$ is taken over a family of second-order linear degenerately elliptic differential operators $L^{\a,\b}$. 

Isaacs equations combine a number of features which are challenging from the numerical point of view. They are fully nonlinear in non-divergence form and the nonlinearity is in general non-convex. Their spatial differential operators may only be degenerately elliptic and often do not exhibit smoothness and structure properties to make linearisation a usable tool. Furthermore, the theory of well-posedness relies, in the setting of viscosity solutions, on comparison principles, which in turn make it natural to impose monotonicity requirements on the numerical methods. Boundary conditions, such as of Dirichlet type, can be posed in multiple non-equivalent forms and comparison principles may rely on a particular form to ensure that they hold.

In this work we present a finite element method to solve time-dependent Isaacs equations with isotropic, possibly degenerate diffusions. We show the convergence to the viscosity solution on Lipschitz domains. Dirichlet boundary conditions are imposed in the viscosity or in the strong sense, or in a combination thereof. The presented analysis is a generalisation of \cite{max_SIAM}, where Bellman equations with homogeneous, strong Dirichlet boundary conditions were approximated.

Compared to the literature for typical quasilinear differential equations, the scientific discourse on the numerical approximation of second-order Isaacs equations consists of a small number of papers. A two-scale finite element method for Isaacs equations was proposed in \cite{isaacs_fem}, where Alexandrov--Bakelman--Pucci estimates are used to prove the convergence with rates of numerical approximations to solutions of uniformly elliptic Isaacs equations on convex domains. Such two-scale finite element methods are closely related to semi-Lagrangian schemes \cite{Li:2018kc}, for which convergence rates for Isaacs operators on unbounded domains were given in \cite{isaacs_sl}, see also the series of works \cite{Jakobsen:2004bl}, \cite{Jakobsen:2006wr}, \cite{Jakobsen:2002kv}, \cite{Jakobsen:2005ci} as well as \cite{Maroso:2006ep} in this context. Finite difference methods for uniformly parabolic Isaacs equations on smooth bounded domains with Dirichlet boundary conditions were shown to converge with rates in \cite{Krylov:2015by} and \cite{Turanova:2015cl}. Deterministic two-player zero-sum games lead to first-order Isaacs equations. We refer for the numerical approximation of deterministic problems to \cite{Biswas:2017ud}, \cite{Botkin:2011fe}, \cite{two_player_numerical}, \cite{isaacs_fv} and the references therein. An important field of application is $H^\infty$ control, for which we highlight \cite{H_inf_num1}, \cite{Ferreira:2010kl}, \cite{H_inf_num2}, \cite{H_inf_isaacs}. 

The main contributions of the paper are the following:

{\em Convergence:} We present a finite element method which is guaranteed to converge to the viscosity solution of the Isaacs equation, including in the degenerate case. Domains only need to be Lipschitz regular.

{\em Boundary conditions:} Dirichlet boundary conditions are treated in the viscosity and strong sense, so that the numerical analysis can follow the requirements arising from the construction of comparison principles.

{\em Barrier functions:} Barrier functions are used to ensure that upper and lower semi-continuous envelopes of the numerical solutions locally satisfy the boundary conditions in the strong sense. The construction of barrier functions is characterized in the degenerate setting and on non-convex domains.

{\em Stability:} Boundary data is mapped onto the approximation space with elliptic projections in order to avoid instabilities, which may arise for instance in the presence of re-entrant corners of the domain.

This paper is organised as follows. In Section \ref{sec:definitions} we specify the class of problems under consideration, introduce a notion of viscosity solution and discuss the two types of Dirichlet boundary conditions. In Section \ref{sec:numerical} we define the numerical scheme. In Section \ref{sec:monotonicity} we establish monotonicity properties; in Section \ref{sec:alg} we show the existence and uniqueness of the numerical solutions; in Section \ref{sec:stablity} we ensure stability; in Section \ref{sec:subsupersolution} we verify that the upper and lower semi-continuous envelopes of numerical solutions are sub- and supersolutions, respectively. In Section \ref{sec:boundary} we return to the different notions of the boundary conditions. In particular, we introduce barrier functions which ensure that boundary conditions can be satisfied in the strong sense at least on a part of the boundary. With that we can state the convergence result in Section \ref{sec:convergence}. In Section \ref{sec:construct_barrier} we show the construction of barrier functions in uniformly parabolic as well as in degenerately parabolic cases. Finally, in Section \ref{sec:experiments} we conclude with numerical experiments verifying the convergence of the scheme and we show an application to a two-player stochastic game.

\section{Isotropic Isaacs equations}
\label{sec:definitions}
We consider Isaacs equations on bounded Lipschitz domains $\O \in \R^d$ with $d \geq 2$. Let $A$ and $B$ be compact metric spaces and $(\a, \b) \in A \times B$. We introduce the linear operators
\begin{equation}\label{eq:linop}
L^{(\a, \b)} : \; H^2(\O) \to L^2(\O), \; w \mapsto - a^{(\a,\b)} \, \Delta w - b^{(\a,\b)} \cdot \nabla w + c^{(\a,\b)} \, w
\end{equation}
and data terms $f^ {(\a,\b)}, v_T \in C(\oO)$, $g \in \{ w \in W^{1,\infty}(\R^d) : \Delta w \in L^\infty(\O) \}$. The mapping
\begin{align*}
 &A \times B \to C(\oO) \times C(\oO, \R^d) \times C(\oO) \times C(\oO),\\ &(\a,\b) \mapsto (a^{ (\alpha,\beta)},  b^ {(\a,\b)}, c^ {(\a,\b)}, f^ {(\a,\b)})
\end{align*}
is assumed to be continuous such that the families of functions $\{ a^{ (\alpha,\beta)} \}_{(\a,\b) \in A \times B}$, $\{ b^{ (\alpha,\beta)} \}_{(\a,\b) \in A \times B}$, $\{ c^{ (\alpha,\beta)} \}_{(\a,\b) \in A \times B}$ and $\{ f^{ (\alpha,\beta)} \}_{(\a,\b) \in A \times B}$ are equicontinuous. Moreover, $a^{(\a,\b)}(x) \geq 0$ for any $(\a, \b) \in A \times B$ and $x \in \O$. It follows that all $L^{\a,\b}$ are degenerate elliptic and that
\begin{equation*}
\sup_{(\a,\b) \in A \times B} \| \, (a^{(\a,\b)},  b^{(\a,\b)}, c^{(\a,\b)}, f^{(\a,\b)}) \, \|_{C(\oO) \times C(\oO, \R^d) \times C(\oO) \times C(\oO)} < \infty.
\end{equation*}
For any sufficiently smooth $w$ we define the Hamiltonian
\[
H w := \adjustlimits \inf_{\b\in B} \sup_{\a\in A} (L^{(\a,\b)} w - f^{(\a,\b)}),
\]
assuming that the supremum and the infimum are applied pointwise. We wish to study the Isaacs problems of the following form:
\begin{subequations}
\label{eq:isaacsibvp}
\begin{alignat}{2}
-\p_t v + H v &=0 	&	&\qquad\text{in }(0,T)\times\O, \label{eq:isaacspde}\\
v&=g 		  	&	&\qquad\text{on }(0,T)\times\pO, \label{eq:isaacsDiri}\\
v&=v_T		  	&	&\qquad\text{on }\{T\}\times\oO,
\end{alignat}
\end{subequations}
where $v_T \in L^\infty(\O)$ and $v_T|_{\pO} = g|_{\pO}$. Even though only the values of $g|_{\pO}$ appear in \eqref{eq:isaacsibvp} and  $v_T|_{\pO} = g|_{\pO}$ it remains valuable to distinguish $g$ and $v_T$: While $v_T$ is only assumed to be uniformly bounded, we require that $g$ can be written as the trace of a function from the more regular space $\{ w \in W^{1,\infty}(\R^d) : \Delta w \in L^\infty(\O) \}$.

Typically there does not exist a smooth solution allowing a classical interpretation of \eqref{eq:isaacsibvp}. Our aim is to construct a finite element method which approximates the viscosity solution of the Isaacs problem \eqref{eq:isaacsibvp}. We now formalise what is meant by a viscosity solution throughout this chapter. Firstly, let us reformulate \eqref{eq:isaacsibvp} as $F=0$ where $F$ is the differential operator
\begin{align*} \label{def:F}
&F(t,x,q,p,r,s) = \nonumber \\ & \left\{ 
\begin{array}{rllll}
\!\!\! - r + {\displaystyle \inf_{\b}\sup_{\a}} \bigl(- a^{(\a,\b)}(x) \, q - b^{(\a,\b)}(x) \!\cdot\! p + c^{(\a,\b)}(x) \, s- f^{(\a, \b)}(x)\bigr) \!&: x \in \O, t < T,\\
s - g(x) \!&: x \in \pO, t < T,\\[2mm]
s - v_T(x) \!&: t = T.
\end{array}
\right.
\end{align*}
Given a bounded function $v: [0,T] \times \oO \to \R$ we define its upper and lower semi-continuous envelopes, respectively, as 
\[
v^*(t,x) := \limsup_{\substack{(s,y) \to (t,x) \\ (s,y) \in [0,T] \times \oO}} v(s,y)
\]
and 
\[
v_*(t,x) := \liminf_{\substack{(s,y) \to (t,x) \\ (s,y) \in [0,T] \times \oO}} v(s,y).
\]
We analogously extend the definition of lower- and upper semicontinuous envelopes to $F$.

Definition \ref{def:vissol_isaacs} below imposes Dirichlet boundary conditions in the viscosity sense on all of $\pO$ in the flavour of the original Barles-Souganidis proof, see also \cite{barles_souganidis,Jensen:2018wn}. For boundary value problems arising from concrete applications, such viscosity boundary conditions may not be strong enough to establish a comparison principle. We therefore introduce a subset $\omega \subset \pO$ on which boundary conditions hold in a stronger sense.

\begin{definition} \label{def:strongBC}
Let $g : \pO \to \R$ and $v_T : \oO \to \R$ be bounded functions. We say that a bounded function $v : [0,T] \times \oO \to \R$ satisfies the Dirichlet boundary conditions $v(t,x) = g(x)$ at $x \in \pO$ {\em strongly} if, for all $t \in [0,T]$,
\[
v^*(t,x) = g(x) \qquad \text{and} \qquad v_*(t,x) = g(x).
\]
Similarly, we say that $v$ satisfies the final time conditions at $x \in \oO$ {\em strongly} if
\[
v^*(T,x) = v_T(x) \qquad \text{and} \qquad v_*(T,x) = v_T(x).
\]
\end{definition}

Depending on the setting $\omega = \emptyset$, $\omega = \pO$ or $\emptyset \subsetneqq \omega \subsetneqq \pO$ may be most appropriate choices.

\begin{definition}\label{def:vissol_isaacs}
A bounded function $v$ is a viscosity supersolution (respectively, subsolution) of \eqref{eq:isaacsibvp} strongly on $\omega \subset \pO$ if, for any test function $\psi \in C^2(\R \times \R^d)$,
\begin{align} \label{eq:vissupsol_isaacs}
F^*(t, x, \Delta \psi(t,x), \nabla \psi(t,x), \partial_t \psi(t,x), v_*(t,x)) \geq 0,
\end{align}
(respectively, 
\begin{align} \label{eq:vissubsol_isaacs}
\left. F_*(t, x, \Delta \psi(t,x), \nabla \psi(t,x), \partial_t \psi(t,x), v^*(t,x)) \leq 0, \right)
\end{align}
provided that $v_{*}-\psi$ attains at $(t,x) \in [0,T) \times \oO$ a local minimum relative to $[0,T) \times \oO$ (respectively, $v^{*}-\psi$ attains a local maximum) and additionally $v_*$ (respectively $v^{*}$) satisfies the final time conditions and the boundary conditions on $\omega$ in the strong sense.
A function which is simultaneously a viscosity super- and subsolution of \eqref{eq:isaacsibvp} is called a viscosity solution.
\end{definition}

The motivation for incorporating strong boundary conditions in the definition of viscosity solutions can be illustrated with a simple example.

\begin{example}
We investigate on $\O = (-1,1)$ the Bellman equation
\begin{align} \label{eq:eikonal}
\sup_{\beta \in \R} \beta \cdot v'(x) - \beta^2 / 4 - (x^2 - x^4)^2 = 0
\end{align}
with homogeneous Dirichlet boundary conditions. Identifying the structure of a Fen\-chel conjugate in the $\beta$-dependent terms, we can rewrite the PDE as $| v'(x) |^2 -(x^2 - x^4)^2 = 0$. Taking the square root, we may simplify the expression to the eikonal equation $| v'(x) | = x^2 - x^4$ with a double-well potential on the right-hand side.

An elementary calculation shows that
\[
v(x) = \frac{|x|^5}{5} - \frac{|x|^3}{3} + \frac{2}{15}
\]
is twice continuously differentiable, solves \eqref{eq:eikonal} classically and satisfies the boundary conditions strongly.

For $c \leq 0$ we define
\[
v_c : \; \oO, \; x \mapsto \begin{cases} 
v(x) \quad & : x \in \O,\\
c \quad & : x \in \pO.
\end{cases}
\]
We initially assume that $\omega$ of Definition \ref{def:vissol_isaacs} is equal to the empty set and show that $v_c$ is in that case a viscosity solution for all $c \leq 0$. In particular the viscosity solution is then not unique as required by the construction in~\cite{barles_souganidis} and there cannot exists a corresponding comparison principle.

We observe that $(v_c)^* = v$ on $\oO$, whereas $(v_c)_* = v_c$ since $v_c$ is lower semi-continuous. Thus it is clear that $v_c$ is a viscosity subsolution of the problem.

We now prove that $v_c$ is also a supersolution. We must show that for all $\psi \in C^2$ such that $(v_c)_* - \psi$ has a local minimum $x$ we have as in \cite[equation (7.10)]{Crandall:1992ta}
\begin{subequations}
\begin{align} \label{eq:eika}
    \tst \sup_\beta \beta \cdot \psi'(x) - \beta^2 / 4 - (x^2 - x^4)^2 \phantom{\}} & \geq 0, && x \in \O,\\ 
    \max\{ (v_c)_*, \tst \sup_\beta \beta \cdot \psi'(x) - \beta^2 / 4 - (x^2 - x^4)^2 \} & \geq 0, && x \in \pO. \label{eq:eikb}
\end{align}
\end{subequations}
Inequality \eqref{eq:eika} holds because $v_c$ is a classical solution on $\O$. Remembering that $x^2 - x^4$ vanishes on $\pO$ and the simplification to the eikonal form, the boundary case \eqref{eq:eikb} can be simplified to 
\begin{align*}
    \max\{ c, | \psi'(x) |^2 \} \geq 0, \qquad x \in \pO,
\end{align*}
which holds trivially. 

Instead, setting $\omega = \pO$ we enforce $(v_c)_* = v_c = 0$ on the boundary, which filters out all the spurious solutions with $c < 0$ and uniqueness is regained. 
\end{example}

Similar examples can also be constructed with second-order equations arising from diffusive problems and in higher dimensions, e.g.~see \cite[Section 2]{Jensen:2018wn}. While the loss of uniqueness of solutions directly interacts with the ability to formulate comparison principles, the flexibility of viscosity solutions to break continuity in the vicinity of the boundary can be helpful in some circumstances, e.g.~when examining singularly perturbed diffusion-advection equations which exhibit a boundary layers for small diffusion coefficients $\eps > 0$ and may fully detach from the Dirichlet data when $\eps \to 0$.

We conclude this section by relating our definition of strong boundary conditions to that in \cite[Section 7.A]{Crandall:1992ta}. There an upper semi-continuous function $v : \oO \to \R$ is a subsolution of the boundary conditions $B(x, v(x), D u(x), D^2 u(x)) = 0$ at $x \in \pO$ if
\begin{align} \label{eq:contBCsub}
B(x, v(x), p, X) \le 0 \quad \text{for} \quad (p, X) \in \overline{J}_{\oO}^{2,+} v(x).
\end{align}
Similarly, a lower semi-continuous $v : \oO \to \R$ is a subsolution of the boundary conditions at $x \in \pO$ if
\begin{align} \label{eq:contBCsup}
B(x, v(x), p, X) \ge 0 \quad \text{for} \quad (p, X) \in \overline{J}_{\oO}^{2,-} v(x).
\end{align}
In our case $B(x, q, p, s) = s - g(x)$ and we refer for the definition of the closed second-order jets $\overline{J}_{\oO}^{2,+} u(x)$, $\overline{J}_{\oO}^{2,-} u(x)$ to \cite{Crandall:1992ta}.

Comparing \eqref{eq:contBCsub} and \eqref{eq:contBCsup} to Definitions \ref{def:strongBC} and \ref{def:vissol_isaacs}, we first of all make an observation which is not directly related to the boundary conditions: we adopt in this paper the notion of bounded, but not necessarily semi-continuous sub- and supersolutions, following \cite{Barles:1988cj} and \cite{barles_souganidis}, see also \cite[Chapter VII]{Fleming:2006tl}. Hence, to bridge this difference, we can adapt the definition of \cite{Crandall:1992ta} to
\begin{align} \label{eq:discontBCsub}
B(x, v^*(x), p, X) \le 0 \quad \text{for} \quad (p, X) \in \overline{J}_{\oO}^{2,+} v(x)
\end{align}
and
\begin{align} \label{eq:discontBCsup}
B(x, v_*(x), p, X) \ge 0 \quad \text{for} \quad (p, X) \in \overline{J}_{\oO}^{2,-} v(x).
\end{align}
Now if the closed second-order jets of $v$ are non-empty, then \eqref{eq:discontBCsub} implies $v^*(x) \le g(x)$ and \eqref{eq:discontBCsup} implies $v_*(x) \ge g(x)$. Since $v^* \ge v_*$ by construction, we arrive at $v^*(x) = v_*(x) = g(x)$, mirroring Definition \ref{def:strongBC}.

\section{The numerical method}
\label{sec:numerical}
We consider a sequence $V_i, i \in \N$, of piecewise linear shape-regular finite element spaces. Let us denote the nodes of the finite element mesh by $\il y$ where $\ell$ is the index over the set of nodes. The associated hat functions are denoted $\il \phi$, i.e.~$\il \phi \in V_i$ such that $ \phi_{i}^{\ell}(y_{i}^{\ell}) = 1$ while  $\il \phi (y_{i}^{s}) = 0$ for all $\ell \neq s$. Set $\hil \phi := \il \phi / \| \il \phi \|_{L^1(\O)}$. Therefore, the $\il \phi$ are normalised in the $L^\infty$ norm whilst the $\hil \phi$ are normalised in the $L^1$ norm. 

Let $V_i^0\subset V_i$ be the subspace of functions which satisfy the homogeneous Dirichlet conditions on $\pO$. Throughout the text we let the index $\ell$  range over the boundary nodes first, in other words, $\il y \in \O$ for $\ell \leq N_i := \dim V_i^0$. 

In order to construct a finite element method which is both stable and consistent, care needs to be taken when imposing {\em non-homogeneous} boundary conditions. Due to the more delicate stability properties of Isaacs operators compared to Bellman operators, on non-convex domains a nodal interpolant is not suitable to map the boundary data onto the approximation space. 

Given $w \in C(H^1(\O))$, we consider therefore linear mappings $\P$ which map $w$ into $V_i$ such that, for all $\hil \phi \in V_i^0$,
\begin{equation}\label{eq:isaacs_ellprojdef}
\langle \nabla \P w, \nabla \hil \phi \rangle = \langle \nabla w, \nabla \hil \phi \rangle.
\end{equation}
\begin{assumption}\label{ass:isaacs_ellproj}
There are linear mappings $\P$ satisfying \eqref{eq:isaacs_ellprojdef} and there is a constant $C\geq 0$ such that for every $ w \in C^{\infty}(\R^{d})$ and $i\in\N$,
\begin{equation}\label{ass:isaacs_ellprojstab}
\norm{\P w}_{W^{1,\infty}(\O)} \leq C \norm{w}_{W^{1,\infty}(\O)}
\quad\text{and}\quad
\lim_{i\tends \infty} \norm{\P w - w}_{W^{1,\infty}(\O)}=0.
\end{equation}
\end{assumption}
It is shown in \cite{approx_prop_fem} that Assumption \ref{ass:isaacs_ellproj} holds if we chose $\P w$ such that it interpolates $w$ on the boundary, \eqref{ass:isaacs_ellproj} is satisfied, $\O$ is a bounded {\em convex polyhedral} domain in $\R^d$, $d \in \{2,3\}$, and the mesh satisfies a local quasi-uniformity condition.

To apply the result for non-convex domains $\O$ and general $w\in C^{\infty}(\R^d)$, we consider a convex polyhedral domain $B$ containing $\O$  and assume there is a locally quasi-uniform mesh on $B$ which coincides with the original mesh on $\O$. Let $\eta$ be a smooth cut-off function with relatively compact support in $B$ such that $\eta \equiv 1$ on $\O$. Then the classical elliptic projection on $B$, acting on $\eta w : B \to \R$, has the required properties. Given this construction for $\P$, it is natural to refer to it as an elliptic projection, see also \cite[Section 4]{max_SIAM}. This construction provides on non-convex domains an approximation of the boundary data, which permits the proof of stability of numerical solutions, see Theorem \ref{thm:num_stable} below, as well as of consistency, see Lemma \ref{lem:ellconsistency}. In conclusion, let $V_i^g\subset V_i$ be the subspace of functions which attain $\P g$ on the boundary. 

The mesh size, i.e. the largest diameter of an element, is denoted $\Delta x_i$. It is assumed that $\Delta x_i \to 0$ as $i \to \infty$. The uniform time step size is denoted $h_i$ with the constraint that $T / h_i \in \N$.  It is assumed that $h_i \to 0$ as $i \to \infty$. Let $\ik s$ be the $k$th time step at the refinement level~$i$. Then the set of time steps is $S_i := \bigl\{ \ik s : k = T/h_i, \dots, 0 \bigr\}$. 

The time derivative is approximated by $d_i$ for which we let the $\ell$th entry of $d_i w( \ik s, \cdot)$ be
\[
(d_i w( \ik s, \cdot))_\ell = \frac{w( \iko s, \il y) - w(\ik s, \il y)}{h_i}.
\]
For the discretisation of $L^{(\a,\b)}$ we allow a splitting into an explicit and an implicit part. For each pair $(\alpha, \b)$ and for each $i$, we introduce the explicit operator $\iab E$ and the implicit operator $\iab I$ such that $L^{(\a,\b)} \approx \iab E + \iab I$ and 
\begin{align*}
\iab E &: \; H^2(\O) \to L^2(\O), \; w \mapsto - \biab a \, \Delta w - \biab b \cdot \nabla w + \biab c \, w,\\
\iab I &: \; H^2(\O) \to L^2(\O), \; w \mapsto - \bbiab a \, \Delta w - \bbiab b \cdot \nabla w + \bbiab c \, w.
\end{align*}
For each $i$ we also seek a discretisation $\iab f$ of $f^{(\a,\b)}$. We now want to make the conceptual statements $L^{(\a,\b)} \approx \iab E + \iab I$ and $f^{(a,b)} \approx \iab f$ more precise.

\begin{assumption} \label{ass:cons}
The splitting is consistent in the sense that
\begin{align*}
\lim_{i \to \infty} \sup_{(\a,\b)\in A \times B} &\bigl( \sup_{0 \le \ell \le N_i} \bigl\| a^{(\a,\b)} - \bigl( \biab a(\il y) +  \bbiab a(\il y) \bigr) \bigr\|_{L^\infty({\rm supp} \, \hil \phi)} \\
& + \bigl\| b^{(\a,\b)} - \bigl( \biab b + \bbiab b \bigr) \bigr\|_{L^\infty(\O,\R^d)} \\
&+   \bigl\| c^{(\a,\b)} - \bigl( \biab c + \bbiab c \bigr) \bigr\|_{L^\infty(\O)} + \bigl\| f^{(\a,\b)} - \iab f \bigr\|_{L^\infty(\O)} \bigr) = 0.
\end{align*}
The coefficients $\biab c$ and $\bbiab c$ are non-negative and that there exists $\gamma \in \R$ such that
\begin{align}\label{react}
 \| \biab c \|_{L^\infty} + \| \bbiab c \|_{L^\infty} \leq \gamma , \qquad \forall\,i\in\N,\;\forall\,(\a,\b)\in A \times B.
\end{align}
The family of mappings
\begin{align*} 
\begin{array}{ll}
A \times B \to L^\infty, \; (\a, \b) \mapsto (\biab a, \biab b, \biab c, \bbiab a, \bbiab b, \bbiab c,f_i^{(\a, \b)})
\end{array} \end{align*}
is continuous for each $i$.
\end{assumption}

The splitting into explicit and implicit part is used to define the internal numerical operators $\siabo E$ and $\siabo I$ as mappings from $H^1(\O)$ to $\R^{\dim V_i}$:
\begin{subequations}\label{eq:discreteop_int}
\begin{align}
(\siabo E w)_\ell & := \biab  a(\il y) \langle \nabla w, \nabla \hil \phi \rangle + \langle -\biab  b \cdot \nabla w + \biab  c \, w, \hil \phi \rangle,\\
(\siabo I w)_{\ell} & := \bbiab a(\il y) \langle \nabla w, \nabla \hil \phi \rangle + \langle -\bbiab b \cdot \nabla w + \bbiab c w, \hil \phi \rangle, \label{sch:imp}\\
(\siabo F)_\ell & := \langle \iab f, \hil \phi \rangle,
\end{align}
\end{subequations}
where $\ell$ ranges over all internal nodes,~i.e. $\ell \leq N_i$. For boundary nodes $\ell > N_i$ we set $(\siabo E w)_\ell = (\siabo I w)_\ell = (\siabo F)_\ell = 0$. Likewise we define the boundary numerical operators $\siabr E$ and $\siabr I$ from $H^1(\O)$ to $\R^{\dim V_i}$:
\begin{subequations}\label{eq:discreteop_bound}
\begin{align}
(\siabr E w)_\ell & := 0\\
(\siabr I w)_{\ell} & := w(\il y)\\
(\siabr F)_\ell & := (\P g)_\ell,
\end{align}
\end{subequations}
where $\ell$ ranges over all boundary nodes,~i.e. $\ell > N_i$. For internal nodes $\ell \leq N_i$ we set $(\siabr E w)_\ell = (\siabr I w)_\ell = (\siabr F)_\ell = 0$. Finally, we obtain the numerical operators $\siab E$ and $\siab I$ by adding the internal and the boundary operators, i.e. $\siab E := \siabo E + \siabr E = \siabo E$, $\siab I := \siabo I + \siabr I$. Similarly, $\siab F := \siabo F + \siabr F$.

When restricted to the domain $V_i$, the numerical operators have matrix representations with respect to the nodal bases $\{ \il \phi \}_\ell$, which we also denote by $\siab E$ and $\siab I$.

Throughout the paper, we make use of the partial ordering of $\RN$: for $x,\,y\in \R^n$, we write $ x\geq y$ if and only if $ x_{\ell} \geq y_{\ell}$ for all $ \ell\in\left\{1,\dots,n\right\}$. For collections $\left\{x^{\a}\right\}_{\a}\subset\R^n$ and $\left\{y^{\b}\right\}_{\b}\subset\R^n$, we define the operators $\sup_{\a}$ and $\inf_{\b}$ componentwise: 
\[
\bigl(\sup_{\a} x^{\a}\bigr)_{\ell}=\sup_{\a}x^{\a}_{\ell} \quad \textrm{and} \quad \bigl(\inf_{\b} y^{\b}\bigr)_{\ell}=\inf_{\b}y^{\b}_{\ell}.
\]
We can now state the numerical scheme for finding an approximate solution to \eqref{eq:isaacsibvp}. We initialise the scheme with the elliptic projection $v_i(T, \cdot) = \P v_T$. Then, in order to find the numerical solution $v_i (\ik s, \cdot) \in V^g_{i}$, we proceed inductively over the remaining time steps $k \in \left\{T/h_i-1, \dots, 0\right\}$:
\begin{align} \label{num}
- d_i v_i(\ik s,\cdot) + \adjustlimits \inf_{\beta \in B} \sup_{\alpha\in A} \bigl( \siab E v_i(\iko s,\cdot) + \siab I v_i(\ik s,\cdot) - \siab F \bigr) = 0.
\end{align}
If all $\siab I$ vanish then \eqref{num} is an explicit scheme, otherwise it is implicit. Note that the terms $d_i v_i$ representing the discretized time derivative vanish on the boundary nodes because $g$ is time independent.

To show the existence and uniqueness of numerical solutions, it is useful to provide an equivalent reformulation of the scheme. Let $w : S_i \times \Omega \to \R$ be a function that satisfies $w(\ik s, \cdot) \in H^1(\O)$ for all $\ik s \in S_i$. Given $\beta$ we denote by $\a_i^{k, \ell, \beta}(w)$ a maximiser of \eqref{eq:maximizer} below. Similarly $\b_i^{k, \ell}(w)$ is a minimiser of \eqref{eq:minimizer}:
\begin{subequations} \label{eq:minmaxcontrols}
\begin{align} \label{eq:maximizer}
& \sup_{\alpha\in A} \left( \siab E w(\iko s,\cdot) + \siab I w(\ik s,\cdot) - \siab F \right)_{\ell},\\
& \inf_{\beta \in B} \left( {\sf E}_i^{(\alpha_i^{k, \ell,\beta}(w),\beta)} w(\iko s,\cdot) + {\sf I}_i^{(\alpha_i^{k, \ell,\beta}(w),\beta)} w(\ik s,\cdot) - {\sf F}_i^{(\alpha_i^{k, \ell,\beta}(w),\beta)} \right)_{\ell}. \label{eq:minimizer}
\end{align}
\end{subequations}
Finally, we write $\a_i^{k, \ell}(w) = \a_i^{k, \ell, \b_i^{k, \ell}(w)}(w)$, i.e.~drop the reference to $\beta$ to imply that $\beta$ is chosen optimally.

Let $\siwk I$ and $\siwk E$ be the matrices whose $\ell$th row at $k$th time step is equal to that of
\[
\si I^{(\alpha_i^{k, \ell}(w), \beta_i^{k, \ell}(w))} \quad \mbox{and} \quad \si E^{(\alpha_i^{k, \ell}(w), \beta_i^{k, \ell}(w))},
\]
respectively. Also let the $\ell$th entry of $\siwk F$ be
\[
\left(\si F^{(\alpha_i^{k, \ell}(w), \beta_i^{k, \ell}(w))}\right)_\ell.
\]
Notice that control $(\alpha_i^{k, \ell}(w), \beta_i^{k, \ell}(w))$ is not necessarily unique. The subsequent analysis is valid regardless of the choice of the control.

We may reformulate the numerical scheme \eqref{num} with $\siwk E$, $\siwk I$ and $\siwk F$ as follows. Initialise the scheme with the elliptic projection so that $v_i(T, \cdot) = P_i v_T$. Then for each $k \in \{T/h_i-1, \dots, 1,0\}$, $v_i \in V_i^g$ solve

\begin{equation}\label{numsol}
(h_i \si I^{k,v_i} + \Id)\, v_i(\ik s,\cdot) + (h_i \si E^{k,v_i} - \Id)\, v_i(\iko s,\cdot) - h_i \si F^{k,v_i} = 0.
\end{equation}

\subsection{Consistency properties of elliptic projections} 
\label{sec:consistency}

We conclude the section with the consistency properties of linear operators for fixed $(\a,\b) \in A \times B$. The result extends \cite{max_SIAM} even in the Bellman setting by including non-homogenous boundary conditions.
To state consistency it is convenient to abbreviate the operator of the numerical scheme as
\begin{align} \label{eq:discreteF}
F_i \, w(\ik s, \il y) := (h_i \si I^{k,w} + \Id)\, w(\ik s,\cdot) + (h_i \si E^{k,w} - \Id)\, w(\iko s,\cdot) - h_i \si F^{k,w}
\end{align}
for $w(\ik s, \cdot) \in V^{g}_i$. Note that despite of the similar notation, the operator $\si F^{k,w}$ has a quite different meaning in that it denotes the forcing term vector, as defined in the first part of this section.

\begin{lemma}\label{lem:ellconsistency}
Let $\psi \in C^2(\R \times \R^d)$, $s_i^{k(i)} \to t \in [0,T)$ and $y_i^{\ell(i)} \to x \in \oO$ as $i\to \infty$. Here $s_i^{k(i)}$ is a time step and $y_i^{\ell(i)}$ a node of the $i$-th refinement. Then
\begin{align} \label{eq:consistency_limsup}
\limsup_{i \to \infty} F_i P_i \psi(s_i^{k(i)}, y_i^{\ell(i)}) \le F^*(t,x, \Delta \psi(t,x), \nabla \psi(t,x), \partial_t \psi(t,x), \psi(t,x))
\end{align}
and
\begin{align} \label{eq:consistency_liminf}
\liminf_{i \to \infty} F_i P_i \psi(s_i^{k(i)}, y_i^{\ell(i)}) \ge F_*(t,x, \Delta \psi(t,x), \nabla \psi(t,x), \partial_t \psi(t,x), \psi(t,x)).
\end{align}
\end{lemma}
\begin{proof}We only prove \eqref{eq:consistency_limsup} since the result for \eqref{eq:consistency_liminf} follows analogously. For ease of notation, the dependence of $k$ and $\ell$ on $i$ is made implicit.

{\em Step 1:} Standard finite difference bounds ensure that
\begin{equation}\label{eq:timederivconv}
\lim_{i\to \infty} d_i P_i \psi(\ik s, \il y) = \p_t \psi(t,x).
\end{equation}

{\em Step 2:} If $\il y \in \O$ then the proof is analogous to the Hamilton--Jacobi--Bellman case described in \cite{max_SIAM}, once the control $\a \in A$ is replaced by the pair of controls $(\a,\b) \in A \times B$. Both Assumptions \ref{ass:isaacs_ellproj} and \ref{ass:cons} are used here.

{\em Step 3:} Now suppose that $\il y \in \pO$. Then it follows from \eqref{ass:isaacs_ellproj} that
\begin{equation}\label{eq:Dirichletconsistency}
\lim_{i\to \infty} F_i P_i \psi(\ik s,\il y) = \psi(t,x) - g(x).
\end{equation}

{\em Step 4:} Consider the sequence $\{ (\ik s, \il y) \}_i$ as specified in the statement of the theorem, in particular with $\il y \in \oO$. We decompose $\{ (\ik s, \il y) \}_i$ into the subsequences of the $(\ik s, \il y)$ where $\il y$ belongs to $\O$ or $\pO$, respectively. Then the conclusions of Steps 1 to 3 above may be applied to the individual subsequences to arrive at \eqref{eq:consistency_limsup}.
\end{proof}

\section{Monotonicity}
\label{sec:monotonicity}

Monotonicity properties of the numerical scheme play a crucial role in ensuring convergence to the viscosity solution.

\begin{definition}\label{def:wDMP}
An operator $ F : V_i \to \R^{N_i} $ is said to satisfy the {\em local monotonicity property} (LMP) if for all $v \in V_i$ such that $v$ has a non-positive local minimum at the internal node $\il y$, we have $( F v )_{\ell} \leq 0$. The operator $F$ satisfies the {\em weak discrete maximum principle} (wDMP) provided that for any $v \in V_i$,
\begin{equation}\label{eq:isaacs_wDMP}
\text{if \,}\bigl( F v \bigr)_{\ell} \geq 0 \text{ for all } \ell \in \{1,\dots, N_i\}, \quad\text{ then } \min_{\O} v \geq \min \{ \min_{\pO} v, 0 \}.
\end{equation}
\end{definition}
Suppose an operator $F$ satisfies the LMP. If $v \in V_i$ has a negative local minimum at an internal node $\il y$, then $((F + \epsilon \Id)v)_\ell < 0$ for any positive $\epsilon$. In particular, the contrapositive of \eqref{eq:isaacs_wDMP} holds. Therefore $F + \epsilon \Id$ satisfies the wDMP for any positive~$\epsilon$.

\begin{assumption} \label{moncon}
For each $(\alpha, \beta) \in A \times B$, we assume that $\siab E$, restricted to $V_i$, has non-positive off-diagonal entries. We assume that $h_i$ is small enough so that $h_i \siab E - \Id$ is monotone for every $(\a, \b)$, i.e.~so that all entries of all $h_i \siab E - \Id$ are non-positive. For each $(\a, \b) \in A \times B$, we suppose that $\siab I$ satisfies the LMP.
\end{assumption}

We now turn to the matrices $ h_i \siwk I+\Id$ and $h_i \siwk E- \Id$ which will later be used in the proof of the well-posedness of the scheme \eqref{numsol}.

\begin{lemma}\label{lem:monotonicity}
Consider a $w : S_i \times \Omega \to \R$ so that $w(\ik s, \cdot) \in H^1(\O)$ for all $\ik s \in S_i$. Then:
\begin{enumerate}
    \item The matrices $h_i \siwk E - \Id$ restricted to $V_i^0$ are monotone.
    \item The mapping $ w \mapsto -(h_i \siwk E - \Id) \,w $ is positive: if $w \geq 0$ then $-(h_i \siwk E - \Id)\, w \ge 0$.
    \item The matrices of $h_i \siwk I + \Id$ restricted to $V_i$ are strictly diagonally dominant M-matrices.
    \item For fixed $w$, the operators $v \mapsto \siwk I v$ and $ v \mapsto (h_i \siwk I + \Id)\,v$ satisfy, respectively, the LMP and wDMP.
\end{enumerate}
\end{lemma}
\begin{proof}
The proof for all but 3.~is analogous to that of \cite[Lemma 2.3]{max_SIAM}, once the control $\a \in A$ is replaced by the pair of controls $(\a,\b) \in A \times B$, using Assumption \ref{moncon}. 

For 3. the strict diagonal dominance of $h_i \siwk I + \Id$ for rows with indices $\ell \leq N_i$ is a direct consequence of its LMP by the same argument as in the case of restriction to $V^{0}_i$ discussed in \cite{max_SIAM}. For indices $\ell > N_i$ rows of matrices $h_i \siwk I + \Id$ are equivalent to that of identity matrix which is strictly diagonally dominant by definition. M-matrix property follows from \cite[Chapter 6, Theorem 2.3, $(M_{35})$]{BER} after noticing that $h_i \siwk I + \Id$ are Z-matrices.
\end{proof}

\subsection{Monotonicity through artificial diffusion}
We show now that, using the method of artificial diffusion one can ensure that $\siab E$ and $\siab I$ satisfy Assumption \ref{moncon}, provided the meshes are strictly acute.

Let $\m T_i$ be the mesh corresponding to the finite element space $V_i$. Given a function $g : \Omega \to \R^d$ and an element $K$ of $\calT_i$, we denote
\[
|g|_K:= \Bigl( \sum_{j=1}^{d} \bigl\| g_j \bigr\|_{L^\infty(K)}^2 \Bigr)^{\half}.
\]
We say that the meshes $\calT_i$ are strictly acute if there exists $\theta \in (0,\pi/2)$ such that for all $i \in \N$:
\begin{equation}\label{eq:acutemesh}
\nabla \il \phi \cdot \nabla \phi_{i}^{l} \bigl|_K \leq - \, \sin(\theta) \; | \nabla \il \phi |_K \; | \nabla \phi^{l}_i |_K \qquad \forall \ell,l \leq N_i,\;\ell \neq l, \; \forall K\in\calT_i.
\end{equation}
Consider a splitting of the form
\begin{align*}
 a^{(\a,\b)} &= \tiab a + \ttiab a, \\
 b^{(\a,\b)} &= \biab b + \bbiab b, \\
 c^{(\a,\b)} &= \biab c + \bbiab c, \\
 f^{(\a,\b)} &= \iab f,
\end{align*}
 where all terms are in $C(\oO)$ or $C(\oO, \R^d)$. Notice the tilde in the notation of the diffusion coefficients. Choose non-negative artificial diffusion coefficients $ \bial \nu $ and $ \bbial \nu $ such that for all $K$ that have $\il y$ as vertex:
\begin{subequations}\label{ineq:nut}
\begin{align}
\bigl( |  \bia b |_K \, +  \Delta x_K \|  \bia c \|_{L^\infty(K)} \bigr) \le \, &  \bial \nu \, \sin(\theta) \, | \nabla \hil \phi |_K \, {\rm vol}(K),\\
\bigl( | \bbia b |_K \, +  \Delta x_K \| \bbia c \|_{L^\infty(K)} \bigr) \le \, & \bbial \nu \, \sin(\theta) \, | \nabla \hil \phi |_K \, {\rm vol}(K).
\end{align}
\end{subequations}
Select $ \biab a $ and $\bbiab a $ both in $C(\oO)$ such that
\begin{subequations}\label{def:ad_coef}
\begin{align}
\biab a(\il y) & \geq \max \bigl\{ \tiab a(\il y), \bial \nu \bigr\},\\
\bbiab a(\il y) & \geq \max \bigl\{ \ttiab a(\il y), \bbial \nu \bigr\}.
\end{align}
\end{subequations}
The splittings of $a^{(\a,\b)} = \biab a + \bbiab a$, $b^{(\a,\b)} = \biab b + \bbiab b$, $c^{(\a,\b)} = \biab c + \bbiab c$ and $f^{(\a,\b)}= \iab f$ can always be chosen so that also Assumption \ref{ass:cons} and \eqref{def:ad_coef} hold. Moreover, it is shown in \cite{max_SIAM} that \eqref{def:ad_coef} implies Assumption \ref{moncon}.

\section{Wellposedness and a solution algorithm}
\label{sec:alg}

\begin{algorithm}[t]
\caption{Howard's method}
\begin{algorithmic}[1]
\Require{$\eta > 0$, \; $k \in \{0,\dots,T/h_i-1\}$, \; $w \in V_i^g$, \; $(\a_0,\b_0) \in A \times B$}
\State $u \gets w$
\State $\b \gets \b_0$
\While{$\| \Psi^{(\a,\b)}(u, w) \| \ge \eta$}
\State $\a \gets \a_0$
\While{$\| \Psi^{(\a,\b)}(u, w) \| \ge \eta$}
\State $u \gets \Phi^{(\a,\b)}(w)$
\State $\a \gets {\rm argmax}_{\a' \in A} \Psi^{(\a',\b)}(u, w)$
\EndWhile
\State $u \gets \Phi^{(\a,\b)}(w)$
\State $\b \gets {\rm argmin}_{\b' \in B} \Psi^{(\a,\b')}(u, w)$
\EndWhile
\State \Return{u}
\end{algorithmic}
\end{algorithm}

In this section we address the existence and uniqueness of numerical solutions and propose a method to compute them. 

\begin{theorem}\label{thm:isaacs_discretewellposedness}
There exists a unique numerical solution $v_i \colon S_i \to V_i^g$ that solves \eqref{num}.
\end{theorem}
\begin{proof}
Fix $k$ and $v_i(s_i^{k+1},\cdot)$. Then \cite[Theorem 5.2]{Howard} shows the existence and uniqueness of a solution $v_i(s_i^k,\cdot)$ of \eqref{num}, noting the continuity of $(\a, \b) \to \siab I$ follows from Assumption \ref{ass:cons} and the monotonicity from Assumption \ref{moncon}. 
\end{proof}

The proof of Theorem 5.2 of \cite{Howard}, which we used here, relies on the exact solution of the inner Bellman equation \eqref{eq:maximizer}, which is reached by a Howard's algorithm in the limit. Using these exact solutions of \eqref{eq:maximizer} in an outer Howard loop yields a sequence whose limit is the solution of~\eqref{num}.

To ensure the completion of the algorithm in finite time we require a termination criterion for the inner and outer loop. In Algorithm 1 (Howard's method) the termination criterion is posed in terms of a tolerance $\eta$. The statement of the algorithm also refers to 
\[
\Psi^{(\a,\b)}(u, w) := (h_i \siab I + \Id) \, u + (h_i \siab E - \Id) w - h_i \siab F,
\]
given $u, w \in V_i^g$ and $(\a,\b) \in A \times B$. Given just $w$, we let $\Phi^{(\a,\b)}(w)$ be the solution $u$ of $\Psi^{(\a,\b)}(u, w) = 0$.

Suppose that $\sum_\ell \eta_\ell < \infty$ and that $w_\ell$ is the function returned by Algorithm 1 with $\eta = \eta_\ell$, $w = v_i(\iko s,\cdot)$ and a fixed $(\a_0,\b_0) \in A \times B$. Then Theorem 5.4 of \cite{Howard} ensures that the $w_\ell$ exist and converge to the unique numerical solution $v_i(s_i^k,\cdot)$ as $\ell \to \infty$.

\section{Stability}
\label{sec:stablity}
For Hamilton--Jacobi--Bellman equations one can bound the value function by the solution of any linear evolution problem associated to a fixed control. In contrast the value function of an Isaacs equation may lie in part above and in part below the solution of such a linear problem. This difference between the Bellman and Isaacs equations extends to the proof of stability of numerical solutions. We begin by adapting \cite[Lemma 3.2]{max_SIAM}.
\begin{lemma}\label{lem:op_stability}
One has $\smnorm{ ( h_i \siw I + \Id )^{-1} }_{\infty} \leq 1$ and $\smnorm{ h_i \siw E - \Id }_{\infty} \leq 1$ for all $i \in \N$ and $w \in V_i$, where the norms are the matrix $\infty$-norms of linear mappings from $V_i$ into $V_i$ and $V_i^0$ into $V_i^0$, respectively.
\end{lemma}
\begin{proof}
Define $v = \sum_{\ell = 1 }^{\dim V_i} \il \phi \equiv 1$, and $v_0 = \sum_{\ell=1}^{N_i} \il \phi \in V_i^0$. By Lemma~\ref{lem:monotonicity}, $h_i \siw I + \Id$ is an invertible M-matrix on $V_i$. Thus, $( h_i \siw I + \Id )^{-1} \geq 0$ entrywise on $V_i$. Thus
\begin{equation}\label{eq:stab_implicit_1}
\begin{split}
\smnorm{ ( h_i \siw I + \Id )^{-1} }_{\infty} 
 &= \max_{1\leq \ell \leq \dim V_i} \sum_{j=1}^{\dim V_i}( h_i \siw I + \Id )^{-1}_{\ell j} = \max_{1\leq \ell \leq \dim V_i} \left(( h_i \siw I + \Id )^{-1} \mathbf{1}\right)_{\ell},
\end{split}
\end{equation}
where $\mathbf{1} \in \R^{\dim V_i}$ is the vector with all entries equal to $1$. 

Since $\nabla v \equiv 0$ (as $v \equiv 1$) we have for each $1\leq \ell \leq N_i$ that
\[
( ( h_i \siw I + \Id ) v )_{\ell} = 1 + h_i \pair{\bar{\bar{c}}_i^{(\alpha_i^\ell(w),\beta^\ell(w))}}{\hil \phi} \geq 1,
\]
where we have used non-negativity of $\bbia c$ and defined $(\alpha_i^\ell(w),\beta^\ell(w))$ analogously to $(\alpha_i^{k, \ell}(w), \beta_i^{k, \ell}(w))$ in \eqref{eq:minmaxcontrols}. Similarly, for each boundary node index $ N_i < \ell \leq \dim V_i$
\[
( ( h_i \siw I + \Id ) v )_{\ell} = 1 + h_i \geq 1.
\]

It follows that $( h_i \siw I + \Id ) v \geq \mathbf{1}$. Combining positive invertibility of $( h_i \siw I + \Id )$ with \eqref{eq:stab_implicit_1}, we conclude that $\smnorm{ ( h_i \siw I + \Id )^{-1} }_{\infty} \leq 1$, as required.

Turning our attention to explicit operators, one has
\[
\smnorm{ h_i \siw E-\Id }_{\infty} = \max_{1\leq \ell \leq N_i} \bigl(- (h_i \siw E - \Id)\,v_0\bigr)_{\ell}
\]
because all entries of the matrix $h_i \siw E - \Id$ are non-positive. For each $1\leq \ell \leq N_i$,  
\[
 \bigl (( h_i \siw E - \Id)\, v \bigr)_{\ell} =\bigl( (h_i \siw E- \Id)\,v_0\bigr)_{\ell} + \sum_{j=N_i+1}^{\dim V_i} \left(h_i \siw E - \Id \right)_{\ell j} \le \bigl( (h_i \siw E- \Id)\,v_0\bigr)_{\ell},
\]
so 
\[
(-(h_i \siw E - \Id)\,v_0)_\ell \leq (- (h_i \siw E - \Id)\, v)_\ell = 1 - h_i \pair{\bar{c}_i^{(\alpha_i^\ell(w),\beta^\ell(w))}}{\hil \phi} \leq 1
\]
because $\bar{c}_i^{(\alpha_i^\ell(w),\beta^\ell(w))} \geq 0$ where $\bar{c}_i^{(\alpha_i^\ell(w),\beta^\ell(w))}$ is defined analogously to $\bar{\bar{c}}_i^{(\alpha_i^\ell(w),\beta^\ell(w))}$ above. Therefore, $ -(h_i \siw E - \Id)v_0 \leq \mathbf{1}$. So $\smnorm{ h_i \siw E - \Id }_{\infty} \leq 1$.
\end{proof}

\begin{theorem}\label{thm:num_stable}
The numerical solutions $v_i$ are uniformly bounded in the $L^{\infty}$ norm: 
\begin{align*}
\norm{v_i}_{L^{\infty}(S_i\times \O)} \leq & \norm{\P v_T}_{L^{\infty}(\O)} + 2 \smnorm{P_i g}_{L^{\infty}(\O)} \\
& + T \sup_{(\a,\b)} \Bigl[ \left( \smnorm{\biab a}_{L^{\infty}(\O)} + \smnorm{\bbiab a}_{L^{\infty}(\O)} \right) \smnorm{\Delta g}_{L^\infty(\O)} \\
& + \left( \smnorm{\biab  b}_{L^{\infty}(\O)}+\smnorm{\bbiab  b}_{L^{\infty}(\O)} \right) \smnorm{\P g}_{W^{1,\infty}(\O)}\\
& + \left( \smnorm{\biab  c}_{L^{\infty}(\O)}+\smnorm{\bbiab  c}_{L^{\infty}(\O)}\right) \smnorm{\P g}_{L^{\infty}(\O)}
 + \smnorm{f_i^{(\a, \b)}}_{L^{\infty}(\O)} \Bigr]. 
\end{align*}
\end{theorem}
\begin{proof}
We split the numerical solution into two parts
\[
v_i(\ik s,\cdot) = v_i^0(\ik s,\cdot) + g_i, \qquad \ik s \in S_i,
\]
where $v_i^0(\ik s,\cdot) \in V_i^0$ and $g_i := \P g \in V_i^g$. Then \eqref{numsol} becomes
\begin{align*}
(h_i \si I^{k,v_i} + \Id)\, (v_i^0(\ik s,\cdot) + g_i) + (h_i \si E^{k,v_i} - \Id)\, (v_i^0(\iko s,\cdot) + g_i) - h_i \si F^{k,v_i} = 0
\end{align*}
or equivalently
\begin{align*}
v_i^0(\ik s,\cdot) = (h_i \si I^{k,v_i} + \Id)^{-1} \bigl(  h_i \si F^{k,v_i} - h_i (\si I^{k,v_i} + \si E^{k,v_i})\, g_i - (h_i \si E^{k,v_i} - \Id)\, v_i^0(\iko s,\cdot) \bigr).
\end{align*}
Using \eqref{eq:isaacs_ellprojdef}, we find
\begin{align*}
\| (\si I^{k,v_i} + \si E^{k,v_i})\, g_i \|_{L^{\infty}(\O)} \le & \, \left(\smnorm{\biab a}_{L^{\infty}(\O)} + \smnorm{\bbiab a}_{L^{\infty}(\O)} \right) \smnorm{\Delta g}_{L^\infty(\O)} \\
& + \left( \smnorm{\biab  b}_{L^{\infty}(\O)}+\smnorm{\bbiab  b}_{L^{\infty}(\O)} \right) \smnorm{g_i}_{W^{1,\infty}(\O)} \\ 
& + \left( \smnorm{\biab  c}_{L^{\infty}(\O)}+\smnorm{\bbiab  c}_{L^{\infty}(\O)}\right) \smnorm{g_i}_{L^{\infty}(\O)}.
\end{align*}
Now an application of Lemma \ref{lem:op_stability} and an inductive argument over the timesteps complete the proof.
\end{proof}

\section{Sub- and supersolutions of the PDE and the viscosity BCs} 
\label{sec:subsupersolution}
Recall the definition of the upper and lower semi-continuous envelopes $v^*$ and $v_*$ of a function $v$ and consider their numerical equivalent defined as follows
\begin{equation}
\label{eq:enum_nvelopes}
\overline{v}(t,x) = \limsup_{i \to \infty} \sup_{(\ik s,\il y) \to (t,x)} v_i(\ik s,\il y), \qquad \underline{v}(t,x) = \liminf_{i \to \infty} \inf_{(\ik s,\il y) \to (t,x)} v_i(\ik s,\il y).
\end{equation}
Owing to Theorem \ref{thm:isaacs_discretewellposedness} and Theorem \ref{thm:num_stable}, $\overline{v}$ and $\underline{v}$ attain finite values. By construction, $\overline{v}$ is upper and $\underline{v}$ lower semi-continuous and $\underline{v} \le \overline{v}$. The proof of the next theorem closely follows that Bellman setting analysed in \cite{max_SIAM}. We outline briefly the main steps to clarify the changes due to the additional $\inf$ operation, control set $B$ and the more general boundary conditions.

At this point we delay the proof that the envelopes satisfy boundary conditions in the strong sense. We can express that by assuming momentarily that $\omega = \emptyset$. The extension to $\omega \neq \emptyset$ follows in the next section.

\begin{lemma}\label{lem:sub_super_solution}
Let $\omega = \emptyset$. Then $\overline{v}$ is a viscosity subsolution of \eqref{eq:isaacsibvp} and $\underline{v}$ is a viscosity supersolution of \eqref{eq:isaacsibvp}.
\end{lemma}

\begin{proof}
We show that $\overline{v}$ is a subsolution. That $\underline{v}$ is a supersolution follows analogously up to a minor asymmetry in the sign of $\gamma \abs{\mu_i}$ in \eqref{eq:isaacs_subsolineq1} below. Suppose that $w \in C^{\infty}(\R \times \R^d)$ is a test function such that $\overline{v}-w$ has a strict local maximum at $ (s,y) \in (0,T)\times\oO$, with $ \overline{v}(s,y)=w(s,y)$. Then following the argument in \cite{max_SIAM} there exists a sequence $\left\{\ik s, \il y \right\}_i$ such that
\begin{align} \label{eq:isaacs_defmu}
v_i(\ik s,\il y)-\P w(\ik s,\il y) \tends \overline{v}(s,y)-w(s,y)=0 
\end{align}
and
\begin{equation}\label{eq:isaacs_defmu2}
v_i(\ika s,\ila y) - \P w (\ika s,\ila y) \leq v_i(\ik s,\il y) - \P w (\ik s,\il y) \; \Leftrightarrow \; \P w (\ika s,\ila y) + \mu_i \geq v_i  (\ika s,\ila y),
\end{equation}
where $\kappa \in \{k,k+1 \}$, $ \ila y \in \supp \hil \phi$ and $\mu_i =  v_i(\ik s,\il y) - \P w (\ik s,\il y)$. Notice that $\mu_i \to 0$ as $i \to \infty$ because of \eqref{eq:isaacs_defmu}.
Similar to \cite{max_SIAM} we conclude from the monotonicity of the scheme \eqref{num} that
\begin{align} \nonumber
0 = \, & - d_i v_i (\ik s,\il y) + \adjustlimits \inf_{\beta \in B} \sup_{\alpha\in A} \left( \siab E v_i(s_i^{k+1},\cdot)+\siab I v_i(\ik s,\cdot) - \siab F \right)_{\ell}  \\ \nonumber
\geq\, & - d_i \P w(\ik s, \il y) +\adjustlimits \inf_{\beta \in B} \sup_{\alpha\in A} \left( \siab E \P w (s_i^{k+1} ,\cdot) + \siab I\P w (\ik s,\cdot) - \siab F \right)_{\ell} \nonumber \\
& - \gamma \abs{\mu_i} \nonumber \\
=\,  & F_i \P w(\ik s, \il y) - \gamma \abs{\mu_i} \label{eq:isaacs_subsolineq1}
\end{align}
Recalling Lemma \ref{lem:ellconsistency}, and $\lim_i \mu_i = 0$, we take the limit $i\tends \infty$ in inequality \eqref{eq:isaacs_subsolineq1} to conclude that
\begin{equation}
0 \geq \liminf_{i \to \infty} F_i \P w(\ik s, \il y) \geq 
F_*(t,y, \Delta w(t,y), \nabla w(t,y), \partial_t w(t,y), w(t,y)) \end{equation}
Hence $\overline{v}$ is a viscosity subsolution.
\end{proof}

\section{Boundary and final time conditions} 
\label{sec:boundary}

A direct consequence of Lemma \ref{lem:sub_super_solution} is that envelopes of the numerical solution satisfy the Dirichlet boundary conditions in the viscosity sense on the entire $\pO$. For the consistency with the strong boundary conditions on $\omega$ we assume the existence of two families of barrier functions $( \zeta_{ y,\epsilon} )_{\epsilon > 0}$ and $( \xi_{y,\epsilon} )_{\epsilon > 0}$ corresponding to the sets of super- and subsolutions respectively. 

\begin{assumption}\label{ass:barrier_functions_existence}
Let $\omega \subseteq \pO$. Let us assume the following:
\begin{enumerate}
\item Family of upper barriers: \\
For all $y \in \omega$, $\eps > 0$ there exists a {\em barrier function} $\zeta_{y,\epsilon} \in W^{1,\infty}([0,T] \times \oO)$ with:
\begin{enumerate}
    \item There exists an $\hat{i} = \hat{i}(\eps) \in \N$ so that for all $i \ge \hat{i}$, $\ik s \in S_i$ there exists a minimiser of $v_i - P_i \zeta_{y,\epsilon}$ over $\{ \ik s\} \times \oO$ which lies on $\{ \ik s \} \times \pO$.
    \item For $t \in [0,T]$ let $q_{t,y,\epsilon}$ be a minimiser of $g-\zeta_{y,\epsilon}$ over $\{ t \} \times \pO$. Then, 
    \[
    \lim_{\epsilon \tends 0} \ q_{t,y,\epsilon} = y \quad \text{and} \quad \lim_{\epsilon \tends 0} \zeta_{y,\epsilon}(t,y) - \zeta_{y,\epsilon} \ (t, q_{t,y,\epsilon}) \geq 0.
    \]
\end{enumerate}
\item Family of lower barriers: \\
For all $y \in \omega$, $\eps > 0$ there exists a {\em barrier function} $\xi_{y,\epsilon} \in W^{1,\infty}([0,T] \times \oO)$ with:
\begin{enumerate}
    \item There exists an $\hat{i} = \hat{i}(\eps) \in \N$ so that for all $i \ge \hat{i}$, $\ik s \in S_i$ there exists a maximiser of $v_i - P_i \xi_{y,\epsilon}$ over $\{ \ik s\} \times \oO$ which lies on $\{ \ik s \} \times \pO$.
    \item For $t \in [0,T]$ let $r_{t,y,\epsilon}$ be a maximiser of $g-\xi_{y,\epsilon}$ over $\{ t \} \times \pO$. Then, 
    \[
    \lim_{\epsilon \tends 0} \ r_{t,y,\epsilon} = y \quad \text{and} \quad \lim_{\epsilon \tends 0} \zeta_{y,\epsilon}(t,y) - \xi_{y,\epsilon} \ (t, r_{t,y,\epsilon}) \leq 0.
    \]
\end{enumerate}
\end{enumerate}
\end{assumption}
In order to understand the connection between convergence of the envelopes to the boundary conditions and the barrier functions let us consider the following example.

\begin{example}\label{counterexample} We study a $1$-dimensional test problem with the homogeneous boundary conditions:
\begin{align*}
- \p_tu + \nabla u -1 & = 0, && \text{in }[0,T)\times(0,1),\\
u & = x, && \text{on }\{T\}\times(0,1).
\end{align*}
In order to find solution we consider a fully implicit numerical scheme with artificial diffusion. In this spirit, the exact solution of $-u_t-\upsilon \Delta u + \nabla u -1 = 0$, interpreting $\upsilon$ as the artificial diffusion coefficient, is:
\[
v_{\upsilon}(t,x)=x -1 + \frac{e^{\frac{x}{\upsilon}}-e^{\frac{1}{\upsilon}}}{1-e^{\frac{1}{\upsilon}}}.
\]
For a fixed $\upsilon$ we can choose $\Delta x$ small enough such that the numerical solution $\tilde{v}_{\upsilon}$ with artificial diffusion $\upsilon$ and uniform mesh with mesh size $\Delta x$ satisfies $\norm{\tilde{v}_{\upsilon}-v_{\upsilon}}_{W^{1, \infty}} < \upsilon$. We note that $v_{\upsilon}$ attains its maximum at
\[
x_{\mathrm{max}} =  \upsilon \log \left[ \upsilon \left( e^{\frac{1}{\upsilon}}-1\right)\right]
\]
with maximal value equal to
\[v_{\upsilon}(t,x_{\mathrm{max}}) = \frac{1}{e^{\frac{1}{\upsilon}} - 1} + \upsilon \log \left[ \upsilon \left( e^{\frac{1}{\upsilon}}-1\right)\right].
\]
In particular, it follows that $x_{\mathrm{max}} \to 1$ and $v_{\upsilon}(t,x_{\mathrm{max}}) \to 1$ as $\upsilon \to 0$. Since
\[
\norm{\tilde{v}_{\upsilon}-v_{\upsilon}}_{W^{1, \infty}} < \upsilon,
\]
we conclude that for a decreasing artificial diffusion coefficient the sequence of numerical solutions $\tilde{v}_{\upsilon}$ has the lower semi-continuous envelope
\begin{equation}\label{uscenvelope}
\underline{v}(t,x)=
\begin{cases}
 x & \quad x \neq 1, \\
0 & \quad  x = 1.
\end{cases}
\end{equation}
However, since $\underline{v}$ has a discontinuity at $1$, there cannot exist a Lipschitz continuous barrier function as described in Assumption \ref{ass:barrier_functions_existence} -- as we decrease $\upsilon$, $v_{\upsilon}$ and hence $\tilde{v}_{\upsilon}$ exhibit an arbitrarily large gradient in the vicinity of the boundary. As a result, for any barrier function $\xi$ there exists $i$ for which the maximum of $\tilde{v}_{\upsilon} - P_i \xi_{y,\epsilon}$ does not lie on the boundary for all $\epsilon > 0$. Note that we can construct both upper and lower barrier functions at $x=0$, e.\,g. $\zeta_0 = x^2$ and $\xi_0 = (x+1)^2$.
\end{example}

\begin{lemma}
Given Assumption \ref{ass:barrier_functions_existence}, we have $\overline{v}(t,x) = \underline{v}(t,x) = g(t,x)$ for all $(t,x) \in [0,T] \times \omega$.
\end{lemma}
\begin{proof}We focus on the case of $v^{*}(t,x)$ as the other case follows analogously. Let $y \in \omega$ and consider a sequence $(\ik s, \il y) \tends (t,y) $ as $i \tends \infty$. We have that
\begin{align}
\limsup_{i \tends \infty} v_i(\ik s, \il y ) =& \lim_{i \tends \infty} P_i\xi_{y,\epsilon}(\ik s, \il y ) + \limsup_{i \tends \infty}[v_i(\ik s, \il y ) - P_i \xi_{y,\epsilon}(\ik s, \il y )] \nonumber  \\ 
 \leq& \xi_{y,\epsilon}(t,y) + \limsup_{i \tends \infty} \sup_{z \in \pO} [v_i(\ik s,z) - P_i\xi_{z,\epsilon}(\ik s,z)] \nonumber \\
=& \xi_{y,\epsilon}(t,y) + \limsup_{i \tends \infty} \sup_{z \in \pO} [\P g(\ik s, z) - P_i\xi_{y,\epsilon}(\ik s,z)]\nonumber \\
=& \xi_{y,\epsilon}(t,y) + \limsup_{i \tends \infty} \sup_{z \in \pO} \Big[\P g(t, z) - P_i\xi_{y,\epsilon}(t,z) \Big.\nonumber \\
&\left. + \left(\P g(\ik s, z) - \P g(t, z) \right) - \left(P_i\xi_{y,\epsilon}(\ik s,z) - P_i\xi_{y,\epsilon}(t,z) \right) \right] \nonumber \\
\stackrel{(a)}{=}&\xi_{y,\epsilon}(t,y) + \limsup_{i \tends \infty} \sup_{z \in \pO} [\P g(t, z) - P_i\xi_{y,\epsilon}(t,z)] \nonumber \\
=&\xi_{y,\epsilon}(t,y) + \limsup_{i \tends \infty} \sup_{z \in \pO}[\left( g(t, z) - \xi_{y,\epsilon}(t,z) \right) \nonumber \\
&+\left( \P g(t, z) - g(t, z) \right) - \left( P_i\xi_{y,\epsilon}(t,z) - \xi_{y,\epsilon}(t,z) \right)] \nonumber \\
\stackrel{(b)}{=}&\xi_{y,\epsilon}(t,y) + \limsup_{i \tends \infty} \sup_{z \in \pO}[ g(t, z) - \xi_{y,\epsilon}(t,z)] \nonumber \\
\stackrel{(c)}{=}&\xi_{y,\epsilon}(t,y) + g(t, r_{t,y,\epsilon}) - \xi_{y,\epsilon}(t, r_{t,y,\epsilon}),
\end{align}
where we get (a) due to Lipschitz continuity of $g$ and $\xi_{y, \epsilon}$ in time, (b) due to $L^{\infty}$ convergence of $\P g(t, z)$ and $P_i\xi_{y,\epsilon}(t,z)$ as $i \to \infty$ and (c) due to the definition of $r_{t,y,\epsilon}$.
By Assumption \ref{ass:barrier_functions_existence} we have that $\lim_{\epsilon \tends 0}\xi_{y,\epsilon}(t,y) - \xi_{y,\epsilon} \ (t, r_{t,y,\epsilon})  \leq 0$ for any $t \in [0,T]$, so we can conclude that
\begin{equation}\label{eq:subsol}\limsup_{i \tends \infty} v_i(\ik s, \il y) \leq g(t,y),
\end{equation}
as $\epsilon \tends 0$.
The proof concludes by completing a similar calculation for \linebreak $\liminf_{i \tends \infty} v_i(\ik s, \il y)$. This gives us
\[g(t,y) \geq \limsup_{i \tends \infty} v_i(\ik s, \il y) \geq \liminf_{i \tends \infty} v_i(\ik s, \il y)  \geq g(t,y),
\]
and the final result follows.
\end{proof}

Finally, also the final time conditions are attained strongly.
\begin{lemma}\label{lem:initalconverge}
The sub and supersolutions $\overline{v}$ and $\underline{v}$ satisfy 
\begin{equation}\label{eq:initialconverge2} 
\overline{v}(T,\cdot)=\underline{v}(T,\cdot)=v_T \quad\text{ on } \, \oO.
\end{equation}
\end{lemma}
\begin{proof}The proof is identical to the Hamilton--Jacobi--Bellman case described in \cite{max_SIAM} once the control $\a \in A$ is replaced by the pair of controls $(\a^{v_i(\ell)},\b^{v_i(\ell)}) \in A \times B$ similarly as in the proof of Lemma \ref{lem:op_stability}, the interpolation operator $\mathcal{I}_i$ is replaced with elliptic projection $\P$ and function space $V_i^{0}$ is replaced with $V_i^{g}$.
\end{proof}

\section{Uniform convergence}
\label{sec:convergence}
Summarising the previous two sections, we have the sub- and supersolution properties of the complete final boundary value problem.
\begin{theorem}
The function $\overline{v}$ is a viscosity subsolution of \eqref{eq:isaacsibvp} and $\underline{v}$ is a viscosity supersolution of \eqref{eq:isaacsibvp}.
\end{theorem}

The last ingredient to establish the convergence of the numerical scheme is a crucial property of the final boundary value problem: a comparison principle. Especially when working with differential operators which degenerate on a part of the domain or classes of singularly perturbed operators, imposing strong Dirichlet conditions on a part of the boundary while viscosity boundary conditions on the remainder is appropriate to guarantee comparison and well-posedness. This is reflected in our analysis in the following assumption in combination with the formulation of Definition~\ref{def:vissol_isaacs}.

\begin{assumption} \label{ass:isaacs_comp}
Let $u$ be a viscosity subsolution and $w$ be a viscosity supersolution. Then $u \le w$ on $\O$.
\end{assumption}

\begin{theorem} \label{thm:isaacs_uniform}
One has $\underline{v} = \overline{v} = v$, where $v$ is the unique viscosity solution of equation \eqref{eq:isaacsibvp}. Furthermore,
\begin{align} \label{eq:conv}
\lim_{i \to \infty} \| v_i - v \|_{L^\infty((0,T) \times \O)} = 0.
\end{align}
\end{theorem}

\begin{proof}
The proof is essentially identical to the Hamilton--Jacobi--Bellman case described in \cite{max_SIAM}.
\hfill \end{proof}

\section{Construction of barrier functions}
\label{sec:construct_barrier}
In order to justify Assumption \ref{ass:barrier_functions_existence} we now present settings for the construction of barrier functions. First, we introduce a method for ensuring the existence of barrier functions for the simpler case of uniformly parabolic operators on a convex domain for fully implicit numerical schemes. After that we show the extension to general IMEX schemes and allow non-convex domains as well as degenerate operators.

We will focus on constructing lower barrier functions $\xi$, since the argument for upper barrier functions $\zeta$ is symmetric and follows by changing the direction of inequalities and exchanging $\sup$ and $\inf$ operators where required.

\subsection{Barrier functions for uniformly parabolic equations on convex domains}
\label{subsec:uniform_par}

We assume the existence of a function $\xi \in \{ w \in W^{1,\infty}([0,T] \times \oO) : \Delta \xi \in L^\infty \}$ which solves 
\begin{subequations} \label{xidef}
\begin{align} \label{xidefPDE}
-\lambda \Delta \xi & = M + \gamma_1 \smnorm{\nabla \xi}_{L^{\infty}(\R^d, \O)} + \gamma_2 \infnorm{\xi} && \text{on } \O,\\
\xi & = g && \text{on } \pO, \label{xidefBC}
\end{align}
\end{subequations}
where
\[
\gamma_1 := \hspace{-2mm} \supiab \hspace{-2mm} \infnorm{\biab b + \bbiab b}, \quad \gamma_2 := \hspace{-2mm} \supiab \hspace{-2mm} \infnorm{\biab c + \bbiab c}
\]
and
\begin{align*}
\lambda := \hspace{-2mm} \infiab \hspace{-2mm} \biab a + \bbiab a, \quad
M := \max \Bigl\{ & \lambda \infnorm{v_T}, \sup_{i, \a, \b} \siab F (\il y) + 1 \Bigr\}.
\end{align*}
The Dirichlet boundary conditions of \eqref{xidefBC} are imposed in the strong sense. Because $g$ does not depend on time it is clear that $\xi$ is constant in $t$. We shall therefore often write $\xi(x)$ instead of $\xi(t,x)$. The function $\xi$ is used as barrier for all $y \in \omega$ and $\eps > 0$: $\xi_{y,\eps} := \xi$.

We assume in this subsection that $\O$ is convex. Then, as outlined below Assumption \ref{ass:isaacs_ellproj}, we can construct $\P$ so that $\P \xi$ interpolates $\xi$ on the boundary. Furthermore, we require for the construction of this section that $\Delta v_T \in L^\infty(\O)$ and $v_T = g$ on $\pO$.

To show that a maximum of
\[
v_i(\ik s,\cdot) - P_i \xi(x)
\]
over ${\ik s} \times \oO$ is attained on the boundary for all ${\ik s}$ we use that the $h_i \sikvi I + \Id$ are M-matrices. Hence it is enough to prove that 
\begin{align} \label{ineq:barrier-numsol}
\left((h_i \sikvi I + \Id)(P_i \xi - v_i(\ik s,\cdot))\right)_{\ell} \geq 0 \qquad \forall \, \ell \leq N_i.
\end{align}

\subsubsection{Fully implicit numerical scheme}

For the sake of clarity, we begin with a fully implicit numerical scheme of the form
\begin{equation}\label{numsol2}
(h_i \si I^{k,v_i} + \Id)\, v_i(\ik s,\cdot) - v_i(\iko s,\cdot) - h_i \si F^{k,v_i} = 0.
\end{equation}
Recall how $(\alpha_i^{k, \ell}(v_i), \beta_i^{k, \ell}(v_i)) \in A \times B$ is a pair of controls such that $\a_i^{k, \ell}(v_i)$ is a maximiser of \eqref{eq:maximizer} and $\b_i^{k, \ell}(v_i)$ is a minimiser of \eqref{eq:minimizer} for $w = v_i$. Noting that $- \langle \Delta \xi, \hil \phi \rangle = \langle \nabla P_i \xi, \nabla \hil \phi \rangle$ is positive due to \eqref{xidef}, we find that
\begin{align*}
&\left((h_i \sikvi I + \Id)(P_i \xi)\right)_{\ell} \nonumber \\
= \;\; &  (P_i \xi)_{\ell} - h_i \, \bbivi a(\il y) \, \langle \Delta \xi, \hil \phi \rangle \nonumber \\ 
&+  h_i \langle -\bbivi b \cdot \nabla P_i \xi + \bbivi c \P \xi, \hil \phi \rangle \nonumber \\
= \;\; & (P_i \xi)_{\ell} - h_i \, \bbivi a(\il y) \, \langle \Delta \xi, \hil \phi \rangle \nonumber \\
& + h_i \langle -\bbivi b \cdot \nabla \xi + \bbivi c \xi, \hil \phi \rangle \nonumber \\
& + h_i \langle -\bbivi b \cdot \nabla ( P_i\xi-\xi) + \bbivi c (\P \xi - \xi) , \hil \phi \rangle \nonumber \\
\geq \;\; &(P_i \xi)_{\ell} + h_i \langle - \lambda \Delta \xi - \gamma_1 \infnorm{\nabla \xi} - \gamma_2 \infnorm{\xi} , \hil \phi \rangle \nonumber \\
& - h_i \left( \gamma_1 \infnorm{\nabla \P \xi - \nabla \xi} + \gamma_2 \infnorm{\P \xi - \xi} \right) \nonumber \\
\stackrel{\eqref{xidef}}{=}  &(P_i \xi)_{\ell} + h_i M - h_i \left( \gamma_1 \infnorm{\nabla \P \xi - \nabla \xi} + \gamma_2 \infnorm{\P \xi - \xi} \right).
\end{align*}
Furthermore, there is a $\hat{i} \in \N$ such that for all $i \geq \hat{i}$.
\[
\sup_i \gamma_1 \infnorm{\nabla \P \xi - \nabla \xi} + \gamma_2 \infnorm{\P \xi - \xi} \le 1
\]
Assuming $i \geq \hat{i}$ for the remainder of the section, and using the definition of $M$
\[
\left((h_i \sikvi I + \Id)(P_i \xi)\right)_{\ell} \geq (\P \xi)_{\ell} + h_i \, \sup_{j, \a, \b} {\sf F}_j^{(\a,\b)} (y_j^\ell).
\]
Recall that $\xi(x) = \xi(\ik s, x) = \xi(\iko s, x)$. With the numerical scheme \eqref{numsol2} we obtain 
\begin{align} \label{eq:implicit_induction}
\left((h_i \sikvi I + \Id)(P_i \xi - v_i(\ik s,\cdot))\right)_{\ell} & \ge (P_i \xi)_{\ell} - v_i(\iko s,\il y).
\end{align}
Because both $(P_i \xi)_{\ell}$ and $v_i(\ik s,\il y)$ interpolate $g$ on $\pO$, it follows $P_i \xi - v_i(\ik s,\cdot) \in V_i^0$. Owing to the M-matrix property of $h_i \sikvi I + \Id$, 
\begin{align} \label{ineq:induction}
    (P_i \xi)_{\ell} - v_i(\iko s,\il y) \geq 0
\end{align}
thus implies condition \eqref{ineq:barrier-numsol}. Hence, by induction, \eqref{ineq:barrier-numsol} holds for all $k$ as soon as \eqref{ineq:induction} is shown at the final time $\iko s = T$.

From \eqref{xidef} we know that $-\lambda \Delta \xi \geq M$. Then, using $v_i(T,\cdot) = \P v_T$ and $\lambda > 0$,
\begin{align*}
\langle \nabla \left( \P \xi^T - v_i(T,\cdot) \right), \nabla \hil \phi \rangle &= - \langle \Delta (\xi - v_T), \hil \phi \rangle\\
& \geq \langle M/\lambda + \Delta v_T, \hil \phi \rangle \geq M/\lambda - \infnorm{\Delta v_T}.
\end{align*}
The definition of $M$ ensures that $\frac{M}{\lambda} \geq \infnorm{\Delta v_T}$. Using M-matrix property of the discrete Laplacian formed with respect to the nodal basis $\{\hil \phi\}_{\ell}$ we obtain \eqref{ineq:induction} at the final time. 

At this point we proved that the maximum of $v_i(\ik s,\cdot) - P_i \xi$ lies on the boundary as required. It remains to show that for $t \in [0,T]$ we can select a maximiser $r_{t,y,\epsilon}$ of $g-\xi_{y,\epsilon}$ over $\{ t \} \times \pO$ such that $\lim_{\epsilon \tends 0} \ r_{t,y,\epsilon} = y$ and $\lim_{\epsilon \tends 0}\xi_{y,\epsilon}(t,y) - \xi_{y,\epsilon} \ (t, r_{t,y,\epsilon})  \leq 0$. Because $\xi$ attains $g$ on whole the boundary, $r_{t,y,\epsilon} := y$ is already such a choice.

\subsubsection{IMEX numerical schemes}
We now extend the above argument to general numerical schemes as defined in \eqref{numsol}. Choosing $\gamma_1, \gamma_2, \lambda$ and $M$ as above, the inequality \eqref{eq:implicit_induction} generalises in the IMEX setting to
\begin{align*} 
\left((h_i \sikvi I + \Id)(P_i \xi - v_i(\ik s,\cdot))\right)_{\ell} & \ge - \left(\! (h_i \sikvi E - \Id) (P_i \xi - v_i(\iko s, \cdot)) \! \right)_{\ell}.
\end{align*}
According to Lemma \ref{lem:monotonicity} we have $- (h_i \sikvi E - \Id) \ge 0$. Therefore, if 
\[
P_i \xi - v_i(\iko s, \cdot)) \geq 0
\]
implies \eqref{ineq:barrier-numsol}. At this point the induction of the previous subsection can be adapted to show that maxima are attained on the boundary. Setting $r_{t,y,\epsilon} := y$ completes the construction.

\subsection{Barrier functions for degenerate equations and general domains}

We now want to remove the two main assumptions of section \ref{subsec:uniform_par}, namely the requirements that the differential operator is uniformly parabolic and that the domain is convex.

We form $\omega$ of the $y$ for which it is possible to ensure the existence of strict supersolutions in the following sense: for all $\eps > 0$ one can find a $\xi_{y,\eps}$ satisfying 
\begin{align} \label{xiye}
\sup_{\a} - a^{(\a,\b)} \, \Delta \xi_{y,\eps} - b^{(\a,\b)} \!\cdot\! \nabla \xi_{y,\eps} & + c^{(\a,\b)} \, \xi_{y,\eps} - f^{(\a, \b)} \ge \eps && \text{on } \O
\end{align}
for all $\b \in C(\O; A)$ as well as
\[
\sup_{\eps > 0} (\| \Delta \xi_{y,\eps} \| + \norm{\xi_{y,\eps} }_{W^{1,\infty}(\O)}) < \infty
\]
and
\begin{subequations} \label{xiyeBC}
\begin{align}
v_T -\xi_{y,\eps} & \le - 2 \eps && \text{on } \oO \setminus B_y(\delta(\eps)), \label{xiyeBCa}\\
v_T - \xi_{y,\eps} & \le - \phantom{2}\eps && \text{on } \oO \cap B_y(\delta(\eps)), \label{xiyeBCb}\\
v_T(y)- \xi_{y,\eps}(y) & > - 2 \eps, \label{xiyeBCc}
\end{align}
\end{subequations}
where $B_y(\delta(\eps))$ is the ball centred at $y$ with radius $\delta(\eps) > 0$, which in turn is a positive parameter depending on $\eps$. Observe that $v_T|_{\pO} = g|_{\pO}$ and therefore \eqref{xiyeBC} also provides control on the boundary, due to $g$ and $\xi_{y,\eps}$ being time independent. It ensures that the maximum of $y -\xi_{y,\eps}$ is attained in the vicinity of $y$ and is non-positive. We can view \eqref{xiye} as generalisation of \eqref{xidefPDE} because $\xi$ of \eqref{xidefPDE} is a strict supersolution of $L^{(\a,\b)} w - f^{(\a,\b)} = 0$ for all $\a \in A$, $\b \in B$. 

In light of Assumptions \ref{ass:isaacs_ellproj} and \ref{ass:cons}, we may choose $\hat{i}$ such that 
\begin{align} \label{ineq:perturb_xi}
C_1(\hat{i}) \, (\| \Delta \xi_{y,\eps} \| + \norm{\xi_{y,\eps} }_{W^{1,\infty}(\O)}) + C_2 \sup_{i \ge \hat{i}} \norm{\P \xi_{y, \eps} - \xi_{y,\eps}}_{W^{1,\infty}(\O)} \leq \eps.
\end{align}
where
\begin{align*}
C_1(\hat{i}) := \sup_{\substack{(\a,\b) \in A \times B \\ i \ge \hat{i}}} &\bigl( \sup_\ell \bigl\| a^{(\a,\b)} - \bigl( \biab a(\il y) +  \bbiab a(\il y) \bigr) \bigr\|_{L^\infty({\rm supp} \, \hil \phi)} \\
& + \bigl\| b^{(\a,\b)} - \bigl( \biab b + \bbiab b \bigr) \bigr\|_{L^\infty(\O,\R^d)} \\
&+   \bigl\| c^{(\a,\b)} - \bigl( \biab c + \bbiab c \bigr) \bigr\|_{L^\infty(\O)} + \bigl\| f^{(\a,\b)} - \iab f \bigr\|_{L^\infty(\O)} \bigr)
\end{align*}
and
\begin{align*}
C_2 := \max \bigl\{ 1, \sup_{\substack{(\a,\b) \\ i \in \N}} \| \biab a +  \bbiab a \|_\infty + \| \biab b +  \bbiab b \|_\infty + \| \biab c +  \bbiab c \|_\infty \bigr\}.
\end{align*}

Recall the definition of $(\alpha_i^{k, \ell}(w), \beta_i^{k, \ell}(w)) \in A \times B$  in \eqref{eq:minmaxcontrols}. For the sake of readability, we write $\hat{\b}$ in place of $\beta_i^{k, \ell}(P_i \xi_{y,\eps} - v_i)$ and $\hat{\a}$ in place of $\a_i^{k, \ell, \hat{\beta}}(P_i \xi_{y,\eps} - v_i)$ in this section, i.e.~$\hat{\a}$ and $\hat{\b}$ are optimal choices when evaluating the numerical operator at $P_i \xi_{y,\eps} - v_i$.

Then the numerical method, applied to $P_i \xi^{k+1} - v_i$ with the control of the infimum frozen at $P_i \xi_{y,\eps} - v_i$, returns at the node $\il y$
\begin{align*}
& \bigl( (h_i \sihahb I + \Id) (P_i \xi_{y,\eps} - v_i(\ik s,\cdot)) + (h_i \sihahb E - \Id) (P_i \xi_{y,\eps} - v_i(\iko s,\cdot)) \bigr)_\ell\\
= \, & h_i \bigl( - (\bihahb a (\il y) + \bbihahb a(\il y)) \langle \Delta \xi_{y,\eps}, \hil \phi \rangle - \langle (\bihahb b + \bbihahb b) \cdot \nabla \P \xi_{y,\eps}, \hil \phi \rangle\\
& \phantom{+ \bigl(}+ \langle (\bihahb c + \bbihahb c) \P \xi_{y,\eps} - \ihahb f, \hil \phi \rangle \bigr)\\
\geq \, & \; h_i \, \langle - \hahb a \Delta \xi_{y,\eps} - \hahb b \cdot \nabla \xi_{y,\eps} + \hahb c \xi_{y,\eps} - \hahb f, \hil \phi \rangle - h_i \, \eps\\
\geq \, & 0.
\end{align*}
We conclude with Lemma \ref{lem:monotonicity} that
\begin{align} \label{ineq:pref_induction_step}
P_i \xi_{y,\eps} - v_i(\iko s,\cdot) \geq 0
\end{align}
implies, for $\ell \leq N_i$,
\begin{align} \nonumber
\bigl( (h_i \sihahb I + \Id) (P_i \xi_{y,\eps} - v_i(\ik s,\cdot)) \bigr)_\ell & \geq \bigl( - (h_i \sihahb E - \Id) (P_i \xi_{y,\eps} - v_i(\iko s,\cdot)) \bigr)_\ell \ge 0.
\end{align}

Because unlike in the case of subsection \ref{subsec:uniform_par} the functions $P_i \xi_{y,\eps} - v_i(\ik s,\cdot)$ do not vanish on the boundary we need to extend this result to the case when $N_i < \ell \leq \dim V_i$. Owing to \eqref{ineq:perturb_xi} and \eqref{xiyeBCb} we know that \eqref{ineq:pref_induction_step} also holds on the boundary for all time steps $\iko s$. Furthermore, $\eqref{xiyeBCb}$ implies that \eqref{ineq:pref_induction_step} is satisfied on all of $\oO$ at the final time $\iko s = T$. In summary we have the induction step that if \eqref{ineq:pref_induction_step} on $\oO$ then $P_i \xi_{y,\eps} - v_i(\ik s,\cdot) \geq 0$ on $\oO$ and the induction base to guarantee that the maximum of
\[
v_i(\ik s,\cdot) - P_i \xi(x)
\]
over ${\ik s} \times \oO$ is attained on the boundary for all ${\ik s}$.

It remains to show that for $t \in [0,T]$ we can select a maximiser $r_{t,y,\epsilon}$ of $g-\xi_{y,\epsilon}$ over $\{ t \} \times \pO$ such that $\lim_{\epsilon \tends 0} \ r_{t,y,\epsilon} = y$ and $\lim_{\epsilon \tends 0}\xi_{y,\epsilon}(t,y) - \xi_{y,\epsilon} \ (t, r_{t,y,\epsilon})  \leq 0$. The former holds because of \eqref{xiyeBCa} and \eqref{xiyeBCc}, while the latter follows from \eqref{xiyeBCb}.

Similarly we assume the existence of strict subsolutions in the following sense: For all $y \in \omega$ and $\eps > 0$ one can find a $\zeta_{y,\eps}$ satisfying 
\[
\inf_{\b} - a^{(\a,\b)} \, \Delta \zeta_{y,\eps} - b^{(\a,\b)} \!\cdot\! \nabla \zeta_{y,\eps} + c^{(\a,\b)} \, \zeta_{y,\eps} - f^{(\a, \b)} \le -\eps \qquad \text{on } \O
\]
for all $\a \in C(\O; A)$ as well as
\[
\sup_{\eps > 0} (\| \Delta \zeta_{y,\eps} \| + \norm{\zeta_{y,\eps} }_{W^{1,\infty}(\O)}) < \infty
\]
and
\begin{align*}
v_T -\zeta_{y,\eps} & \ge 2 \eps && \text{on } \oO \setminus B_y(\delta(\eps)),\\
v_T - \zeta_{y,\eps} & \ge \phantom{2}\eps && \text{on } \oO \cap B_y(\delta(\eps)),\\
v_T(y)- \zeta_{y,\eps}(y) & < 2 \eps. \label{xiyeBCc}
\end{align*}

\section{Numerical experiments} \label{sec:experiments}
In this section we present two numerical experiments showing the viability of the presented method. In the first experiment, we analyse convergence rates of a fully nonlinear second order Isaacs problem with a known solution and we confirm at least linear convergence in $L^2$, $L^{\infty}$ and $H^1$ norms. In the second experiment we calculate the value function of a stochastic tag-chase game with asymmetric velocities and vanishing diffusion on a non-convex domain. The presented finite element method was implemented in Python with FEniCS \cite{LoggMardalEtAl2012a} and is available from the public repository \cite{github} under the GNU Lesser General Public License.

\subsection{Isaacs problem with exact solution}
Let the spatial domain $\O$ in $\R^2$ be the equilateral triangle with vertices $(\pm \sqrt{3},\ \frac{1}{2})$ and $(0,\-1)$. We will study the following Isaacs problem:
\begin{equation}\label{exp:Is1}
-v_t - \inf_{\b \in [\frac{1}{4},\ \frac{1}{2}]} \Bigl\{ \b \sqrt{\frac{x^2+y^2}{T-t+1}} \Delta v+\frac{1}{2}\frac{1}{\sqrt{T-t+1}}\abs{\nabla v} \Bigr\} = -\frac{1}{2}\frac{\sqrt{x^2+y^2}}{(T-t+1)^{3/2}}.
\end{equation}
Equation \eqref{exp:Is1} is indeed of Isaacs type because the Euclidean norm of the gradient may be written alternatively as
\[\abs{\nabla v} = \sup_{\{\a \in \R^2 : \abs{\a}=1\}}\{\a \cdot \nabla v\}.
\]
One can now verify through direct calculation that the function
\begin{align} \label{ex1:sol}
v(x,y,t) = \exp\left(-\sqrt{\frac{x^2+y^2}{T-t+1}}\right)+\sqrt{\frac{x^2+y^2}{T-t+1}}
\end{align}
solves \eqref{exp:Is1} exactly. The final and boundary data is also given by the right-hand side of \eqref{ex1:sol}. 

We split the scheme so that the advection term is treated explicitly. Note that in order to ensure the monotonicity of the scheme we may also assign part of the diffusion to the explicit operator. In order to improve the rate of convergence this was done locally, \ie on a node-wise basis. In case the naturally occurring diffusion at the node is not sufficient, we introduce artificial diffusion. Overall, this approach leads to artificial diffusion near the origin where the differential operator is degenerately elliptic, while for the majority of nodes of the mesh it is zero. We choose the largest timestep guaranteeing the monotonicity of the method which leads to $O(h_i)=O(\Delta x_i)$. Any amount of the diffusion left after ensuring the monotonicity of the explicit operators is treated implicitly.

The convergence of the scheme is reflected on Figure \ref{fig:isaacs_error_plot} which plots errors at $t=0$ for different mesh sizes. 
The rates of convergence are as follows:
\begin{center}
\small \begin{tabular}{c | c c | c c | c c} 
 $\Delta x$ & $L^{\infty}$ & Rate & $L^{2}$ & Rate & $H^1$ & Rate \\ [0.5ex] 
 \hline
0.4330 & 1.364\text{e-}02 & 0.66 & 7.062\text{e-}03 & 0.96 & 6.792\text{e-}02 & 0.91 \\
0.2165 & 1.040\text{e-}02 & 0.91 & 3.695\text{e-}03 & 0.78 & 3.737\text{e-}02 & 0.98 \\
0.1083 & 5.715\text{e-}03 & 1.01 & 2.361\text{e-}03 & 0.93 & 1.899\text{e-}02 & 1.01 \\
0.0541 & 2.838\text{e-}03 & 1.07 & 1.266\text{e-}03 & 1.02 & 9.391\text{e-}03 & 1.03 \\
0.0271 & 1.330\text{e-}03 & 1.10 & 6.234\text{e-}04 & 1.06 & 4.581\text{e-}03 & 1.03  \\
0.0135 & 6.034\text{e-}04 & 1.11 & 2.941\text{e-}04 & 1.07 & 2.223\text{e-}03 & 1.03  \\
0.0068 & 2.708e-04 & & 1.374e-04 & & 1.082e-03 \\
[1ex] 
\end{tabular}
\end{center}

\begin{figure}[t]
\begin{center}
\begin{tikzpicture}[scale=0.9]
\begin{loglogaxis}[xlabel = $\textrm{Inverse of mesh size} \ 1/\Delta x$,
ylabel=Error,grid=both,major grid style={black!50}]
\addplot [color=black,mark=x] coordinates {  
    (2.309401076748307, 0.013639266366533676)
    (4.618802153495687, 0.010401435920023294)
    (9.237604306990917, 0.005715015650753497)
    (18.475208613975084, 0.0028381367067700225)
    (36.950417227932824, 0.0013298142339577268)
    (73.90083445584372, 0.0006034492847517559)
    (147.80166891149, 0.0002708195785323664)	
};
\addlegendentry{$L^{\infty}$}

\addplot [color=black,mark=o] coordinates {
    (2.309401076748307, 0.007061530848818077)
    (4.618802153495687, 0.0036946469556077403)
    (9.237604306990917, 0.0023607423346638937)
    (18.475208613975084, 0.0012662600120144404)
    (36.950417227932824, 0.000623378871009345)
    (73.90083445584372, 0.00029407795593793066)
    (147.80166891149, 0.0001373826384599483)
};
\addlegendentry{$L^2$}

\addplot [color=black,mark=square] coordinates {
    (2.309401076748307, 0.0679169001743772)
    (4.618802153495687, 0.03736683894962015)
    (9.237604306990917, 0.01898968371348606)
    (18.475208613975084, 0.009390610646586662)
    (36.950417227932824, 0.004580707884902971)
    (73.90083445584372, 0.002222797240397769)
    (147.80166891149, 0.0010824123226126667)
};
\addlegendentry{$H^1$}
\end{loglogaxis}
\end{tikzpicture}
\caption{Approximation error of Experiment 1}
\label{fig:isaacs_error_plot}
\end{center}
\end{figure}
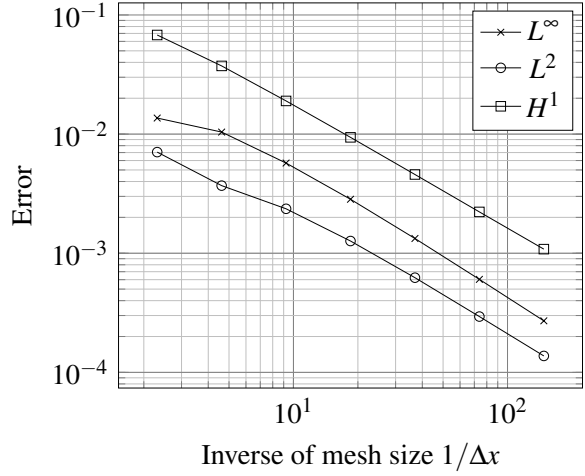

\subsection{Tag-Chase Game with random noise}
Imagine two players moving on the $\R^2$ plane. One player is the pursuer (we will denote them by P) and tries to catch the other one, who is the evader (we will denote them by E). Both of them are allowed to choose their direction freely: $\a, \b \in [-\pi,\pi]$ and $\a$ is the direction chosen by the evader, while $\b$ is the direction chosen by the pursuer. The pursuer moves with speed $v^{\b}_P$ while the evader moves with speed $v^{\a}_E$. In particular, the respective speeds are functions of the angles. Additionally, some choices of a direction for the pursuer and the evader may be subject to a random noise behaving like a standard Brownian motion. Having specified the setting we are able to formulate the dynamics explicitly as
\begin{align*}
dx_{1}^{P} &= (v^{\b}_P)_1 \, \sin(\b) dt + \sigma^{\b}_{P}(x_{1}^{P}) \, d W_{(1,P)},\\
dx_{2}^{P} &= (v^{\b}_P)_2 \, \cos(\b) dt + \sigma^{\b}_{P}(x_{2}^{P}) \, d W_{(2,P)},\\
dx_{1}^{E} &= (v^{\a}_E)_1 \, \sin(\a) dt + \sigma^{\a}_{E}(x_{1}^{E}) \, d W_{(1,E)},\\
dx_{2}^{E} &= (v^{\a}_E)_2 \, \cos(\a) dt + \sigma^{\a}_{E}(x_{2}^{E}) \, d W_{(2,E)}.
\end{align*}
Additionally, we assume that
\begin{align*}
W_P &= (W_{(1,P)}, W_{(2,P)}),\\
W_E &= (W_{(1,E)}, W_{(2,E)}),
\end{align*}
are two $\R^2$-valued, mutually independent standard Wiener processes. We reduce this $4$-dimensional problem to a $2$-dimensional one by allowing the origin of the coordinate system to move along with the pursuer. In this case our dynamics are
\begin{align*}
d\hat{x}_1 & = ((v^{\b}_P)_1 \, \sin(\b) - (v^{\a}_E)_1 \, \sin(\a)) \, dt + \sigma^{(\a, \b)}(\hat{x}_1) \, d\hat{W}_{1},\\
d{\hat{x}_2} & = ((v^{\b}_P)_2 \, \cos(\b) - (v^{\a}_E)_2 \, \cos(\a)) \, dt + \sigma^{(\a, \b)}(\hat{x}_2) \, d\hat{W}_{2},
\end{align*}
where $\sigma^{(\a, \b)}(\hat{x}_i) = \sqrt{\sigma^{\b}_P(x^P_i))^2 + (\sigma^{\a}_E(x^E_i))^2}$ and $\hat{W}$ is an $\R^2$-valued standard Wiener process satisfying
\[
\sigma^{(\a,\b)}(\hat{x}_i) \hat{W}_i(t) = \sigma^{\b}_{P}(x_{i}^{P}) W_{i,P}(t) - \sigma^{\a}_{E}(x_{i}^{E}) W_{i,E}(t), \quad t \geq 0,\;i=1,2.
\]

The pursuer catches the evader (and thus wins the game) if they manage to reduce their distance from the evader to some value $r \in \R$. The evader wins the game when they manage to increase the distance to pursuer to some given $R > r$ or if they manage to avoid the capture before some time $T$. Note that in this case the spatial domain of the problem becomes $\overline{\O} := \overline{B}(0,R) \setminus B(0,r)$. 

Mathematically, the evader receives the pay-out of $1$ whenever they win the game and receive $0$ otherwise. We write for the expected pay-out to the evader $\mathcal{J}(\vec{\a},\vec{\b},x,t)$, where $\vec{\a}: [0,T] \to A$, $\vec{\b}: [0,T] \to B$ are functions of time so that $\vec{\a}(t)$, $\vec{\b}(t)$ are the controls chosen by the evader and pursuer at time $t$, respectively.

The value function
\[
v(x,t) := \adjustlimits \inf_\beta \sup_\alpha \mathcal{J}(\a,\b,x,t)
\]
then solves the second-order Isaacs equation
\begin{align*}
-\p_t v  + \inf_\beta \sup_\alpha \big(- a^{(\a, \b)} \Delta v - b^{(\a,\b)} \cdot \nabla v \big) = 0 \qquad & \text{on } [0,T) \times \O,\\
v = 1 \qquad & \text{on } [0,T) \times \{ x : \| x \| = R\} ,\\[2mm]
v = 0 \qquad & \text{on } [0,T) \times \{ x : \| x \| = r\} ,\\[2mm]
v = 1 \qquad & \text{on } \{T\} \times \oO,
\end{align*}
where 
\[
a^{(\a,\b)} := \frac{(\sigma^{(\a, \b)})^2}{2}, \qquad 
b^{\a,\b} := \begin{pmatrix} 
    (v_P^{\b})_1 \, \sin(\b) - (v_E^{\a})_1 \, \sin(\a) \\
    (v_P^{\b})_2 \, \cos(\b) - (v_E^{\a})_2 \, \cos(\a)
\end{pmatrix}.
\]
We shall assume that the pursuer is faster than the evader when moving in the horizontal direction. Moreover, we assume that the diffusion in the upper part of the domain scales with the vertical position $x_2$, while it is constant where $x_2 \le 0.1$. Specifically, this corresponds to the following choice of the diffusion and advection coefficients:
\begin{align*}
a^{(\a,\b)} = \max\{x_2,0.1\}, \qquad
b^{(\a,\b)} &=
\begin{pmatrix} 
    4 \sin(\b) - 0.5 \sin(\a) \\
    \cos(\b) - \cos(\a)
\end{pmatrix}.
\end{align*}
The numerical approximation of the value function is displayed in Figure \ref{fig:asymmetric_stochatic_chase}. Note the asymmetric nature of the graph, due to the different speeds of the pursuer and the evader in the horizontal direction as well as a larger effect of the stochastic component of the equation in the upper part of the domain.

\begin{figure}[t]
\centering
\includegraphics[width=\textwidth, trim={15cm 2.5cm 10cm 2.5cm},clip]{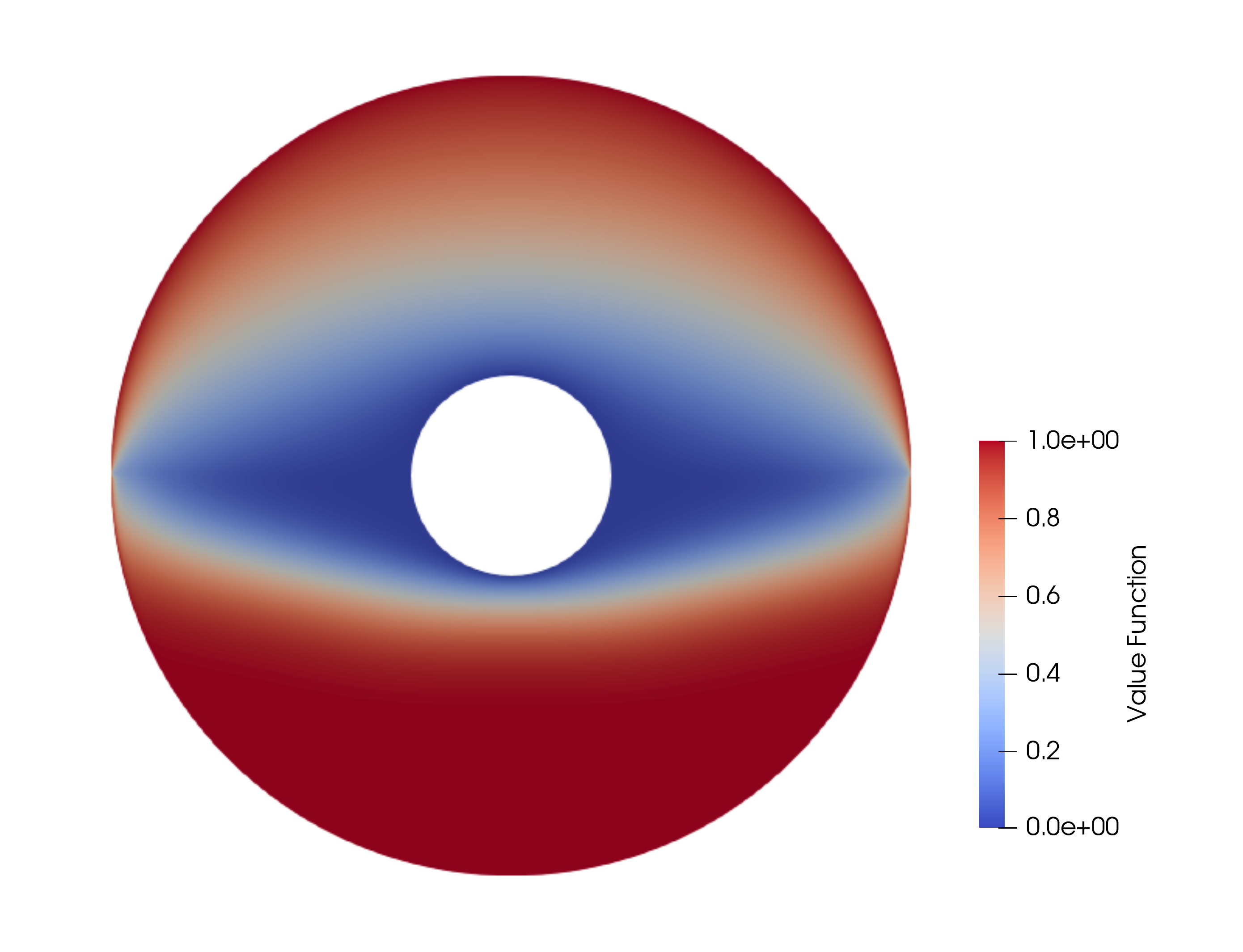}
\caption{The value function at $t=0$ for asymmetric velocities and the diffusion having a lower value in the lower part of $\O$.}
\label{fig:asymmetric_stochatic_chase}
\end{figure}

\section*{Declarations}

{\em Research funding:} Bartosz Jaroszkowski acknowledges the support of the EPSRC grant 1816514. Max Jensen acknowledges the support of the Dr Perry James Browne Research Centre.

\noindent {\em Conflicts of interest:} The authors declare that they have no competing interest.

\noindent {\em Code availability:} At \url{https://doi.org/10.5281/zenodo.4598310}.

\newcommand{\etalchar}[1]{$^{#1}$}


\end{document}